\documentclass[12pt]{article}

\usepackage{color}
\usepackage{amsmath}
\usepackage{amsthm}
\usepackage{bbm}
\usepackage{hyperref}
\usepackage{cleveref}
\usepackage{float}

\usepackage{amsmath}
\usepackage{amssymb}
\usepackage{hyperref}
\hypersetup{
    colorlinks=true,
    linkcolor=blue,
    filecolor=magenta,
    urlcolor=cyan,
}

\usepackage{bbm}

\newtheorem{theorem}{Theorem}

\newtheorem{claim}{Claim}

\newtheorem{lemma}{Lemma}

\newtheorem{proposition}{Proposition}

\newtheorem{corollary}{Corollary}

\def\a{\alpha}
\def\b{\beta}
\def\g{\gamma}

\def\A{\mathcal{A}}
\def\G{\Gamma_x}
\def\R{\mathbb{R}}
\def\N{\mathbb{N}}
\def\d{\delta}
\def\o{\omega}
\def\O{\Omega}
\def\cavu{(\text{Cav}\, u)}

\def\p{\pi_M}
\def\dk{\Delta(K)}
\def\k{\kappa}
\def\l{\ell}
\def\R{\mathbb{R}}
\def\s{\sigma}
\def\T{\mathcal{T}}
\def\t{\tau}
\def\x{\xi}

\def\argmax{\text{arg\, max}}

\newcommand{\headings}[1]{\vskip .3cm \noindent \textbf {#1} \hskip .4cm }
\newcommand{\ignore}[1]{}

\usepackage{graphicx}

\begin{document}
\title{On the Monotonicity and Rate of Convergence of the Markovian Persuasion Value}


\author{Dimitry Shaiderman\footnote{Department of Mathematics, Hebrew University of Jerusalem, Jerusalem 9190401, Israel, email: dima.shaiderman@gmail.com. The author wishes to thank the anonymous referees and AE for their detailed comments and suggestions, which were of great help in advancing this work and improving the manuscript. The author also wishes to thank Ehud Lehrer, Eilon Solan, and Andrei Iaboc for their valuable comments and remarks.} }


\date{\today}
\maketitle
\thispagestyle{empty}

\begin{abstract}
{\footnotesize We study a dynamic Bayesian persuasion model called Markovian persuasion. In such a model, the belief of the receiver regarding the current state of a Markov chain $(X_n)_{n\geq 1}$, over a finite state space $K$, is controlled through signals she obtains from a sender, who observes $(X_n)_{n\geq 1}$ in real time. At each stage $n\geq 1$, the receiver takes an action based on his current belief, which together with the realized state of $X_n$, determines the $n$'th stage payoff of the sender. The sender's goal in a Markovian persuasion game is to find a signaling policy that maximizes her expected $\delta$-discounted sum of stage payoffs for a discount factor $\d \in [0,1)$. We show that starting from any invariant distribution $(X_n)_{n\geq 1}$ the trajectory of the $\d$-discounted value is a monotone decreasing in $\d$. By combining this result with the opposite increasing monotone trajectories found in Lehrer and S.\ (2025, \cite{MPwSR}), we are able to derive an upper bound on the rate of convergence of the $\d$-discounted values (as $\d \to 1^-$) in the case where $(X_n)_{n\geq 1}$ is ergodic. The results for the Markovian persuasion model are then extended to the Markov chain games model of Renault (2006, \cite{Renault}). 
}

\vskip .2cm

\noindent{\footnotesize \textbf{MSC2000 Subject Classification:} 91A25, 91A27, 91A28}
\end{abstract}

{
\large 
\itshape 

To Israel and Raisa Poliachenko,

}

\section{Introduction}

\textbf{The Markovian persuasion game.}
The study of information provision is central in many economic fields and models. In the field of Bayesian persuasion, one studies the optimal way in which an informed party, the sender, should manipulate the beliefs of a decision maker, the receiver, using her private information. The baseline model in the field, introduced and studied by Gentzkow and Kamenica \cite{Kamenica}, concentrates on a static situation, where the sender transmits to the receiver a stochastic signal, whose correlation with the possible states of nature determines the amount of information the sender wishes to share. By committing to her information provision policy, the sender allows the receiver to update his belief over the possible states upon obtaining the sender's realized signal. Based on his updated belief, the receiver takes an action, which together with the true state determines the payoffs of both parties. A famous example of such a game, given in \cite{Kamenica}, is that where a prosecutor, who is fully informed about the state of nature, i.e., whether the suspect is guilty or innocent, wishes to persuade a judge, for instance, to convict the suspect of a crime.     

A natural extension of the baseline model of Gentzkow and Kamenica, concerns the case where the sender and the receiver engage in a repeated game, in which the true state evolves over time stochastically according to a discrete-time Markov chain $(X_n)_{n\geq 1}$. This dynamic Bayesian persuasion problem was studied by Ely \cite{Ely}, Renault, Solan, and Vieille \cite{Solan}, Farhadi and Teneketzis \cite{Farhadi}, and Lehrer and Shaiderman \cite{MP, MPwSR}.
Compared to the static model of Gentzkow and Kamenica \cite{Kamenica}, where the sender acts myopically, in the dynamic model the information provision policy of the sender affects both her current stage payoff as well as her future stage payoffs.

Optimal information provision policies in dynamic Bayesian persuasion games were described by Ely \cite{Ely}, Renault, Solan, and Vieille \cite{Solan}, Farhadi and Teneketzis \cite{Farhadi}, for stylized problems, involving special families of Markov chains and payoff functions.

As in Renault, Solan, and Vieille \cite{Solan} and Lehrer and Shaiderman \cite{MP,MPwSR}, we assume that the sender discounts her stage payoffs by a discount factor $\delta \in [0,1)$. The optimization problem of the sender is thus to find an information provision policy that maximizes her expected $\delta$-discounted sum of stage payoffs. Under such optimal policy, the sender's payoff is termed the \textit{$\delta$-discounted value}. The dynamic Bayesian persuasion problem at hand will be referred to as a Markovian persuasion game.

\textbf{The main results.}
Our first main result, stated in Theorem \ref{Thm1}, shows that in the Markovian persuasion game, whenever the Markov chain $(X_n)_{n\geq 1}$ is stationary--i.e., starts from an invariant distribution of the chain, the $\delta$-discounted values decrease as $\delta$ grows. Qualitatively, such result states that under stationary dynamics, the sender cannot benefit from patience. 

The result of Theorem \ref{Thm1} adds to recent results of Lehrer and Shaiderman \cite{MPwSR} regarding the analytic behavior of the $\d$-discounted value, as $\d$ varies. As shown in \cite{MPwSR}, when $(X_n)_{n\geq 1}$ is irreducible, there exist weighted averages of the $\d$-discounted values evaluated at various sets of initial distributions for $(X_n)_{n\geq 1}$, that increase in $\d$. In particular, there exist such initial distributions, for which the $\d$-discounted values at those distributions, do not decrease with $\d$. Such phenomena, together with the result of Theorem \ref{Thm1}, show that patience may have opposite effects on the sender, based on the initial distribution of the Markov chain. The exact details of those opposing effects is described in detail in Subsection \ref{SubSec Survey Mon. Traj.}, dedicated to the survey of known monotone trajectories of the $\d$-discounted values in $\d$. 

Our second main result, stated in Theorem \ref{Thm3}, utilizes the results on monotone trajectories in $\d$ in the case where $(X_n)_{n\geq 1}$ is ergodic (i.e., irreducible and aperiodic), to provide an upper bound on the rate of convergence of the $\d$-discounted values as $\d$ tends to $1$. This theorem shows, that regardless of the initial distribution of $(X_n)_{n\geq 1}$, the $\d$-discounted value converges to a number, termed the `asymptotic value', at a rate which is faster then $(1-\d)^{\a}$ for every $\a \in (0,1)$.

\textbf{The monotonicity through information design technique.}
The recent results on monotone trajectories of the $\d$-discounted values presented in the current work and that of Lehrer and Shaiderman \cite{MPwSR} were derived based on a new proof technique, which we term `monotonicity through information design'. Such technique consists of generalizing the original Markovian persuasion model, by adding an exogenous independent Bernoulli process, having an effect on the evolution of information in the baseline model. 

The first instance of such technique was showcased in Lehrer and Shaiderman \cite{MPwSR}, where the true state was revealed to the receiver at each stage with a fixed positive probability, called the revelation rate, independently of all other events in the game. In particular, in such a setup, the sender is no longer the informational monopolist. The monotone trajectories found by Lehrer and Shaiderman (e.g., Theorems 2 and 3 in \cite{MPwSR}) were derived by analyzing the asymptotics of the $\d$-discounted values of the generalized game as $\delta$ approaches $1$, and comparing different revelation rates. 

In the current work, the information design used to derive Theorem \ref{Thm1}, is one in which, at any stage, the receiver's memory of past signals from the sender is erased with a fixed positive probability. The exogenous Bernoulli process in that case, is the one that governs the memory erasures along the game. Such information design is accompanied with an introduction of an auxiliary player, which also has the capability to erase the receiver's memory. The result of Theorem \ref{Thm1} is then made possible by embedding the generalized model with memory erasures into a zero-sum stochastic game with signals between the sender and the adversary (see Subsection \ref{Subsec Stochastic Game with Signals}). Thus, the combination of zero-sum stochastic games with signals together with the monotonicity through information design technique can be viewed as a further advancement and development of the original technique used in Lehrer and Shaiderman \cite{MPwSR}.

\textbf{Connection to zero-sum repeated games with incomplete information.}
As it turns out, the results of Theorems \ref{Thm1} and \ref{Thm3}, hold also for the $\d$-discounted values of another model of repeated games. Such model, first introduced and studied by Renault \cite{Renault} forms a dynamic extension of Aumann and Maschler's seminal model of repeated games with incomplete information on one-side \cite{Aumann}. The proof of the parallel  results to those of Theorems \ref{Thm1} and \ref{Thm3}, which are summarized in Theorem \ref{Thm RG-1,3}, are achieved using the above techniques, with appropriate modifications for the current model.

As it pertains to the contribution of Theorem \ref{Thm RG-1,3} to the respective field, we have two remarks. First, Sorin (see Proposition 3.11 in \cite{Sorin}) proved that the $\d$-discounted values of Aumann and Maschler's model decrease with $\d$. Thus, Theorem \ref{Thm RG-1,3} can be viewed as a verification of Sorin's result to the dynamic setup of Renault \cite{Renault}, in the case where the dynamics are stationary. Secondly, Theorem \ref{Thm RG-1,3} shows that under ergodic dynamics, the rate of convergence of the $\d$-discounted values admits a `faster' upper bound compared to the bound known for the Aumann and Maschler model, which is of order $\sqrt{1-\delta}$ (e.g., Corollary V.2.9.\ on p. 223 in \cite{MSZ}). 

Lastly, we leave it open for future research, to study whether the methodology and results of the current work can be utilized to obtain similar results for  models of zero-sum repeated games with incomplete information on both sides. In particular, does there exist similar monotone trajectories and upper bounds for the rates of convergence in the splitting game model of Laraki \cite{Laraki} and in the model of Gensbittel and Renault \cite{Gensbittel}. 

\textbf{Additional related literature.}
In terms of the monotonicity properties of values of repeated games, it is interesting to mention that the phenomena described in the current work need not hold in similar models to the ones considered in the paper. Indeed, under $n$-stage evaluations rather than discounted ones, the $n$-stage values of repeated games with incomplete information on one side with signals need not follow any monotone pattern (see Lehrer \cite{Leh1987} and Yariv \cite{Yariv}). 

In a broader sense, the paper adds to existing results regarding the analytic behavior of the $\d$-discounted values as a function of $\d$, for different game-theoretic models. In Shapley's model of stochastic games (e.g., Shapley \cite{Shapley}), the $\d$-discounted values are known to be a semi-algebraic function of $\d$ (e.g., Bewely and Kohlberg \cite{Kohlberg}). The latter, in particular, implies that the $\d$-discounted values converge as the discount factor approaches $1$. Nevertheless, in the more general model of stochastic games with signals (e.g., Chapter 4.3 in Solan and Ziliotto \cite{SolanZiliotto}), the behavior of the $\d$-discounted values need not be regular in $\delta$. Ziliotto \cite{Bruno} provided an example of a zero-sum repeated game with public signals and perfect observation of the actions, where the $\d$-discounted values do not converge as the discount factor approaches $1$. Such phenomenon, i.e., failure of convergence, was also shown to exist under the framework of zero-sum stochastic game with compact action sets as showcased by Vigeral \cite{Vigeral} and Ziliotto \cite{Bruno}.  Moreover, Sorin and Vigeral \cite{Sorin_Vigeral} study the oscillations of the $\d$-discounted values in the above setups.

\textbf{Organization of the paper.}
Section \ref{Sec-Model} is devoted to the formal description of the Markovian persuasion game.  The main results of the paper and corresponding discussions can be found in Section \ref{Sec-Main Res}. In Section \ref{Sec Toolbox} we develop the technical toolbox necessary for the proofs of the main results. Sections \ref{Sec Toolbox} and \ref{Sec Proof of Thm3} are devoted to the proofs of Theorems \ref{Thm1} and \ref{Thm3}, respectively. In Section \ref{Sec RG proof} we survey the modifications to the preceding proofs required to establish the main results on the model of dynamic zero-sum repeated games with incomplete information on one-side, which are described in Subsection \ref{Subsec Repeated Game}. Lastly, in Appendix \ref{Appendix A} we calculate and show several examples to the main results of the paper.

\section{The Markovian Persuasion Model}\label{Sec-Model}

\subsection{The Utilities of the Sender and the Receiver}

Let $K = \{1,...,k\}$ be a finite set of possible states. The \textit{receiver} is a player equipped with a set of actions $B$. We assume that $B$ is either finite or $B = [0,1] \subset \R$. The utility of the receiver is given by a mapping $W:K\times B \to \R$. The receiver is assumed to maximize his utility based on his belief's distribution over $K$. Formally, given the belief $\xi = (\xi^{\l})_{\l \in K} \in \dk$ the receiver takes an action 
    \begin{align*}
        b(\xi) \in \text{arg\, max}_{b \in B} \sum_{\l \in K} \xi^{\l} \cdot W(\l,b).
    \end{align*}

The sender's utility is described by a mapping $V:K\times B \to \R$. That is, it depends on the true state as well as on the receiver's action. For simplicity we assume that $V(\cdot, \cdot)$ is non-negative. A common assumption in the literature (e.g.\ \cite{Kamenica}) asserts that the receiver `breaks ties' in favor of the sender. That is, for every belief $\xi \in \dk$ it holds that
        \begin{align*}
            \sum_{\l \in K} \xi^{\l} \cdot V(\l,b(\xi)) \geq  \sum_{\l \in K} \xi^{\l} \cdot V(\l,b^*)
        \end{align*}
        for all $b^* \in \text{arg\, max}_{b \in B} \sum_{\l \in K} \xi^{\l} \cdot W(\l,b)$. The action policy $b(\cdot):\dk \to \R$ of the receiver is assumed to be known by the sender. 
        
\bigskip

\subsection{The Dynamics of the True State}

Let $(X_n)_{n \geq 1}$ be a Markov chain over $K$, described by a prior probability $q \in \Delta(K)$ and transition rule given by a stochastic matrix $M$. The receiver knows that the true state of the game evolves stochastically over time according to $(X_n)_{n\geq 1}$. The receiver has incomplete information regarding the evolution of the true state over time, as he does not observe the realizations of $(X_n)_{n \geq 1}$, but rather is only aware of the distribution of $(X_n)_{n \geq 1}$ (i.e., of $q$ and $M$).

\bigskip

\subsection{The Sender's Information Provision}

The \textit{sender} is an agent who seeks to influence the actions of the receiver by filtering his information regarding $(X_n)_{n\geq 1}$ in real time. At each time period $n\geq 1$, the sender observes the current true state (i.e., the realization of $X_n$). Upon obtaining this information the sender sends a signal $s_n$ to the receiver, from a finite set of signals $S$, where $|S| \geq |K|$. 

A \emph{signaling policy} $\sigma$ of the sender is described by a sequence of stage strategies $(\s_n)_{n\geq 1}$, where $\s_n : (K \times S)^{n-1}\times K \to \Delta(S)$. That is, the signal $s_n$ is chosen according to the lottery $\s_n$, whose laws may take into account past events, i.e., the realizations of $X_1,...,X_{n-1}$, previously sent signals $s_1,...,s_{n-1}$, together with `today's information', being the  realization of $X_n$. Denote by $\Sigma$ the space of all signaling policies available to the sender.

\bigskip

\subsection{The Receiver's Beliefs Processes}

At the outset of the game, the sender commits to a signaling policy $\s \in \Sigma$, which is known to the receiver. Such a commitment allows the receiver to assign probabilities to different events along the game. Formally, by Kolmogorov's Extension Theorem, the strategy $\s$ together with $(X_n)_{n \geq 1}$ induce a unique probability measure $P_{q,\s}$ on the space of infinite histories of states and signals $(K \times S)^{\N}$.

The receiver's information regarding the evolution of the true state is described by two stochastic processes, termed the \textit{priors} and \textit{posteriors} processes. Formally, for each $n\geq 1$, the $n$'th prior $q_n = (q_n^{\l})_{\l \in K} \in \dk$, is defined by
\begin{equation*}
q_n^{\l} := P_{q,\s}\left(X_n=\l \,|\, s_1,...,s_{n-1}\right),
\end{equation*}
so that $q_1=q$, whereas the $n$'th posterior $p_n = (p_n^{\l})_{\l \in K} \in \dk$ is given by
\begin{equation*}
p_n^{\l} := P_{q,\s}\left(X_n=\l \,|\, s_1,...,s_{n-1},s_n\right).
\end{equation*} 
In simple terms, $q_n$ describes the receiver's belief on the state of the Markov chain at time $n$ before obtaining the $n$'th signal $s_n$, whereas $p_n$ describes the Bayesian update of $q_n$ after obtaining $s_n$. 

Two important properties showcasing the connection of the processes $(q_n)_{n\geq 1}$ and $(p_n)_{n\geq 1}$ are 
\begin{equation}\label{Martingale}
    E_{q,\s} \, (p_{n}\,|\, q_n) = q_n, \quad \forall n\geq 1,
\end{equation}
and 
\begin{equation}\label{M-shift}
    q_{n+1} = p_n M, \quad \forall n\geq 1.
\end{equation}
Note that \eqref{Martingale} follows from the martingale property, whereas \eqref{M-shift} follows from the Markovian dynamics of $(X_n)_{n\geq 1}$. 

\subsection{The Timeline of the Markovian Persuasion Game}

At the outset of the game, the sender commits to a signaling policy $\s \in \Sigma$. The game is then played in stages, so that at each stage $n\geq 1$ the following events take place in chronological order:
\begin{itemize}
    \item The sender sends the signal $s_n$, which is chosen by the lottery $\s_n$.
    \item The receiver updates his prior belief $q_n$ to his posterior belief $p_n$.
    \item The receiver takes the action $b(p_n)$.
    \item The sender's and the receiver's obtain their $n$'th stage payoffs, which equal $V(X_n,b(p_n))$ and $W(X_n,b(p_n))$, respectively.
    \item The receiver computes his next prior $q_{n+1} = p_n M$, and the game proceeds to stage $n+1$. 
\end{itemize}

\subsection{The Sender's Optimization Problem}

The sender's goal is to find a policy $\s \in \Sigma$ that maximizes her expected $\delta$-discounted sum of payoffs, where $\delta \in [0,1)$ is a fixed discount factor. That is, her goal is to maximize 
    \begin{align*}
     \gamma_{\delta}(q,\s) & := E_{q,\s}\left[(1-\d)\sum_{n=1}^{\infty} \d^{n-1} \cdot V(X_n, b(p_n))\right]\\
     & = E_{q,\s}\left[(1-\d)\sum_{n=1}^{\infty} \d^{n-1}\cdot u(p_n) \right]
    \end{align*}
    over all $\s \in \Sigma$, where $u:\dk \to \R_+$ is given by $u(\xi) := \sum_{\l \in K} \xi^{\l} \cdot V(\l,b(\xi))$. A key property of the function $u$ is that its upper semi-continuous. Such property is guaranteed by the tie-breaking rule of the receiver. Upper semi-continuous functions play a key role in the subsequent analysis due to the fact that they attain their maximum on compact domains. 
    
    Let $v_{\d} (q) := \sup_{\s \in \Sigma} \gamma_{\delta}(q,\s)$ denote the $\d$-discounted value of the Markovian persuasion game with prior $q \in \dk$. Note that as $V(\cdot,\cdot) \geq 0$, we have that $v_{\d}(\cdot)\geq 0$ as well. Note that for $\delta=0$, the optimization problem coincides with the classical Gentzkow and Kamenica \cite{Kamenica} one-shot Bayesian persuasion problem. The solution in that case is described in terms of the concavification of $u$. Formally, set
\begin{align*}
(\text{Cav}\, u)(\xi) := \inf \{h(\xi)\,:\, h:\dk \to \mathbb{R} \text{\, is concave}, u \leq h \}, \quad \forall \xi \in \dk.
\end{align*}
That is, $(\text{Cav}\, u)$ is the smallest concave function that majorizes $u$. Then, Gentzkow and Kamenica \cite{Kamenica} have shown that $v_{0}(q) = \cavu (q)$ for every $q \in \dk$. For the sake of completeness, the steps leading to the former result will be described in detail in Subsection \ref{SubSec Opt Str}.

\section{Main Results}\label{Sec-Main Res}

\subsection{The Effect of Stationary Dynamics}

The first main result of the current work establishes a new monotonicity property of the discounted value $v_{\delta}$ of a Markovian persuasion game, and reads as follows:

\begin{theorem}\label{Thm1}
Consider a Markovian persuasion game. Then, for every invariant distribution $\pi$ of $M$, the trajectory $\delta \mapsto v_{\delta}(\pi)$ is non-increasing on $[0,1)$.
\end{theorem}

From a dynamics standpoint, Theorem \ref{Thm1} states that whenever the Markov chain $(X_n)_{n\geq 1}$ is stationary, the $\delta$-discounted value cannot increase when $\delta$ grows. In behavioral terms, this implies that increased patience does not benefit the sender in Markovian persuasion games with stationary dynamics.

The result of Theorem \ref{Thm1} may look intuitive for the following reason. When the patience of the sender increases, i.e., when the discount factor $\d$ grows, the relative weight of future payoffs has a higher effect on the total payoff. Thus, in that case, it seems intuitive that the sender does not benefit from the informational spillover across time. Theorem \ref{Thm1} affirms that intuition when the prior is an invariant distribution. However, as we shall see in Theorem \ref{Thm2} in Subsection \ref{SubSec Survey Mon. Traj.} such intuition does not fit all priors.

\bigskip

\headings{A Glimpse to the Proof Technique of Theorem \ref{Thm1}}

The technique we use to prove Theorem \ref{Thm1} can be coined `monotonicity through information design'. The main idea behind this technique, which was first showcased in Lehrer and Shaiderman \cite{MPwSR}, is to add an information parameter $x \in (0,1)$ to the Markovian persuasion model, and study a corresponding generalized strategic model that depends on $x$. 

The generalized strategic model suited to prove Theorem \ref{Thm1} is a zero-sum stochastic game with signals (see Subsection \ref{Subsec Stochastic Game with Signals}) between the sender and an auxiliary adversary. In this game, the information parameter $x \in (0,1)$ arises by considering a situation in which at each stage $n\geq 1$ the receiver's memory of past observed signals $s_1,...,s_{n-1}$ is erased by chance with probability $x$, independently of all other events along the game. The auxiliary adversary to the sender introduced to the stochastic game is a player who only has recall of his own actions and the current stage period. At each stage, conditional on the receiver's memory not being erased by chance, the adversary possesses the ability to erase it.

The memory erasures suit the current framework, as the stationary dynamics guarantee that the belief of the receiver after any memory erasure will return to the invariant distribution $\pi$ of the Markov chain. Therefore, as from the sender's view point any memory erasure starts the game anew, she can be strategically myopic in every sub-game occurring between any two successive memory erasures. The monotonicity arises from a `coupling' argument; if the adversary decides to erase the receiver's memory at each stage with prob.\ $(y-x)/(1-x)$, for some $x<y<1$, the sender essentially faces a situation where the memory erasures are performed according to an i.i.d.\ Bernoulli process with mean $y$ rather then with mean $x$, as under chance's process. 

The suitability of exogenous i.i.d.\ Bernoulli processes is useful due to their connection to $\delta$-discounted valuations. Indeed, under both chance's process, and the adversary's coupling process, the random duration between any two memory erasures follows a geometric distribution with parameters $x$ and $y$, respectively. As it turns out, the expected sum of payoffs along any such random duration admits a formula in terms of the discounted expected payoff (see Subsection \ref{Subsec Random Duration}). Such a fact is rooted in the known probabilistic interpretation of discounted valuations, in which the total reward of a player equals his payoff of the last stage of the game, where the former is drawn at random from a geometric distribution. The formal details of the proof of Theorem \ref{Thm1} can be found in Section \ref{Sec-Proof of Thm 1}. 

\bigskip

\subsection{A Survey of Monotonic $\delta$-Trajectories}\label{SubSec Survey Mon. Traj.}

The monotonic $\delta$-trajectory $\d \mapsto v_{\d}(\p)$ exhibited in Theorem \ref{Thm1} adds to a family of monotonic $\delta$-trajectories for the Markocian persuasion model showcased in Lehrer and Shaiderman \cite{MPwSR} for an irreducible $M$. However, contrary to the decreasing nature of the $\delta$-trajectories described in Theorem \ref{Thm1}, 
Theorem 3 in \cite{MPwSR} provides the following family of increasing $\delta$-trajectories.

\begin{theorem}[Lehrer and Shaiderman \cite{MPwSR}]\label{Thm2}
Consider a Markovian persuasion game where $M$ is irreducible and let $\p$ be the unique invariant distribution of $M$. Assume that the beliefs $\{\xi_1,...,\xi_k\} \subseteq \Delta(K)$ satisfy 
    \begin{itemize}
        \item[\emph{(a)}]  $\xi_1,...,\xi_k$ are affinely independent.
        \item[\emph{(b)}] $M(\Delta(K)) \subseteq conv\,\{\xi_1,...,\xi_k\}$, i.e., for every $q \in \Delta(K)$, $qM \in conv\,\{\xi_1,...,\xi_k\}$.
    \end{itemize}
    Then, the trajectory $$\delta \mapsto \gamma_1 v_{\delta}(\xi_1) + ... + \gamma_k v_{\delta}(\xi_k),$$ is non-decreasing on $[0,1)$, where $\gamma_1,...,\gamma_k$ are the unique convex weights satisfying
    \begin{equation*}
        \pi_M = \gamma_1 \xi_1 + ... + \gamma_k  \xi_k.
    \end{equation*}
\end{theorem}

In simple terms, Theorem \ref{Thm2} states that there exists infinitely many families of priors, such that certain averages of their $\d$-discounted values exhibit a  monotonic non-decreasing behavior as $\d$ grows. This result is weaker than that of Theorem \ref{Thm1}, in the sense that it does not allow one to deduce the existence of a non-decreasing monotonic trajectory of the form $\d \mapsto v_{\d} (q)$ for some prior $q \in \dk$. Nevertheless, it implies that there is no range of discount factors in which all the trajectories $\d \mapsto v_{\d}(\xi_1),..., \delta \mapsto v_{\d}(\xi_k)$ decrease simultaneously.

\bigskip

The informational parameter $x \in (0,1)$ added to Markovian persuasion model in order to prove Theorem \ref{Thm2} serves to embody random information revelations. In further detail, for every $n\geq 2$, some information regarding $X_n$ is assumed to be revealed to the receiver at the start of the $n$'th stage with probability $x$.\footnote{In exact terms, at the start of the $n$'th stage, prior to obtaining the $n$'th sender's signal $s_n$, the receiver's prior belief $q_n$ is `split' to the beliefs $\{\xi_1,...,\xi_k\}$. Such a `split' is well defined as by Eq.\ \eqref{M-shift} $q_n \in M(\dk)$ for every $n\geq 2$. Further details can be found in Section 9 in \cite{MPwSR}.} Thus, in comparison to the proof technique of Theorem \ref{Thm1}, we have a situation where the receiver obtains information as opposed to losing all of his past information.

Without elaborating on the different steps required to prove Theorem \ref{Thm2}, we  note that the monotonicity in Theorem \ref{Thm2} builds on a coupling argument as well. The corresponding coupling is done by the sender, who can at each stage reveal to the receiver the corresponding information regarding $X_n$ with prob.\ $(y-x)/(1-x)$, where $x<y<1$, whenever the stochastic revelation failed to do so. This `coupling' is intended to mimic the case where the stochastic revelations occur with probability $y$.

\subsubsection{On the Analytic Behavior of Two Trajectories}

For any irreducible stochastic matrix $M$, a set of priors $\xi_1,...,\xi_k$ satisfying the assumptions of Theorem \ref{Thm2} consists of the Dirac measures on $\dk$, denoted by $\textbf{e}_1,...,\textbf{e}_k$, so that $\textbf{e}_{\l}$ assigns probability mass $1$ to state $\l \in K$. Thus, Theorem \ref{Thm2} implies that the mapping $\Psi:[0,1) \to \R_{+}$ defined by 
\begin{align*}
    \Psi(\d) = \p^1 \cdot v_{\d}(\textbf{e}_1)+\cdots + \p^k \cdot v_{\d}(\textbf{e}_k)
\end{align*}
is non-decreasing on $[0,1)$. Also, let $\Phi: [0,1) \to \R_{+}$ be defined by $\Phi(\delta) := v_{\delta}(\pi_M)$, where we recall that $\pi_M$ denotes the unique invariant distribution of the irreducible stochastic matrix $M$. By Theorem \ref{Thm1} we have that $\Phi$ is non-increasing on $[0,1)$. We shall now exhibit some phenomena relating the analytic behavior of $\Psi$ and $\Phi$.

As we shall showcase in Subsection \ref{SubSec Opt Str}, $v_{\delta}(\cdot):\dk \to \R_{+}$ is concave for every $\delta \in [0,1)$. Therefore, by definition, $\Phi \geq \Psi$ on $[0,1)$. The latter, together with Theorems \ref{Thm1} and \ref{Thm2} implies the following corollary.

\begin{corollary}\label{Cor1}
Assume that there exists some $\delta_{0} \in [0,1)$ such that $\Phi(\delta_0) = \Psi(\delta_0)$. Then, both $\Phi$ and $\Psi$ are constant and equal to $\Phi(\delta_0)$ on the interval $[\delta_0,1)$.
\end{corollary}

Note that as $\Phi(0) = \cavu (\pi_M)$ and 
\begin{align*}
\Psi(0) & = \pi_M^1\cdot \cavu(\textbf{e}_1)+\cdots+\pi_M^k\cdot \cavu (\textbf{e}_k)\\
& = \pi_M^1\cdot u(\textbf{e}_1)+\cdots+\pi_M^k\cdot u (\textbf{e}_k),
\end{align*} Theorems \ref{Thm1} and \ref{Thm2} together with Corollary \ref{Cor1}, imply that the analytic relation between the graphs of $\Phi$ and $\Psi$ must obey one of the cases depicted in Figure \ref{figure 1}. In Appendix \ref{Appendix A} we will give examples showing that each of the cases is possible. Nevertheless, the existence of an example for Case B in which $0<\d_0<1$ (where $\d_0$ was described in Corollary \ref{Cor1}) remains an open question.

\begin{figure} 
\centering
\includegraphics[scale=0.4]{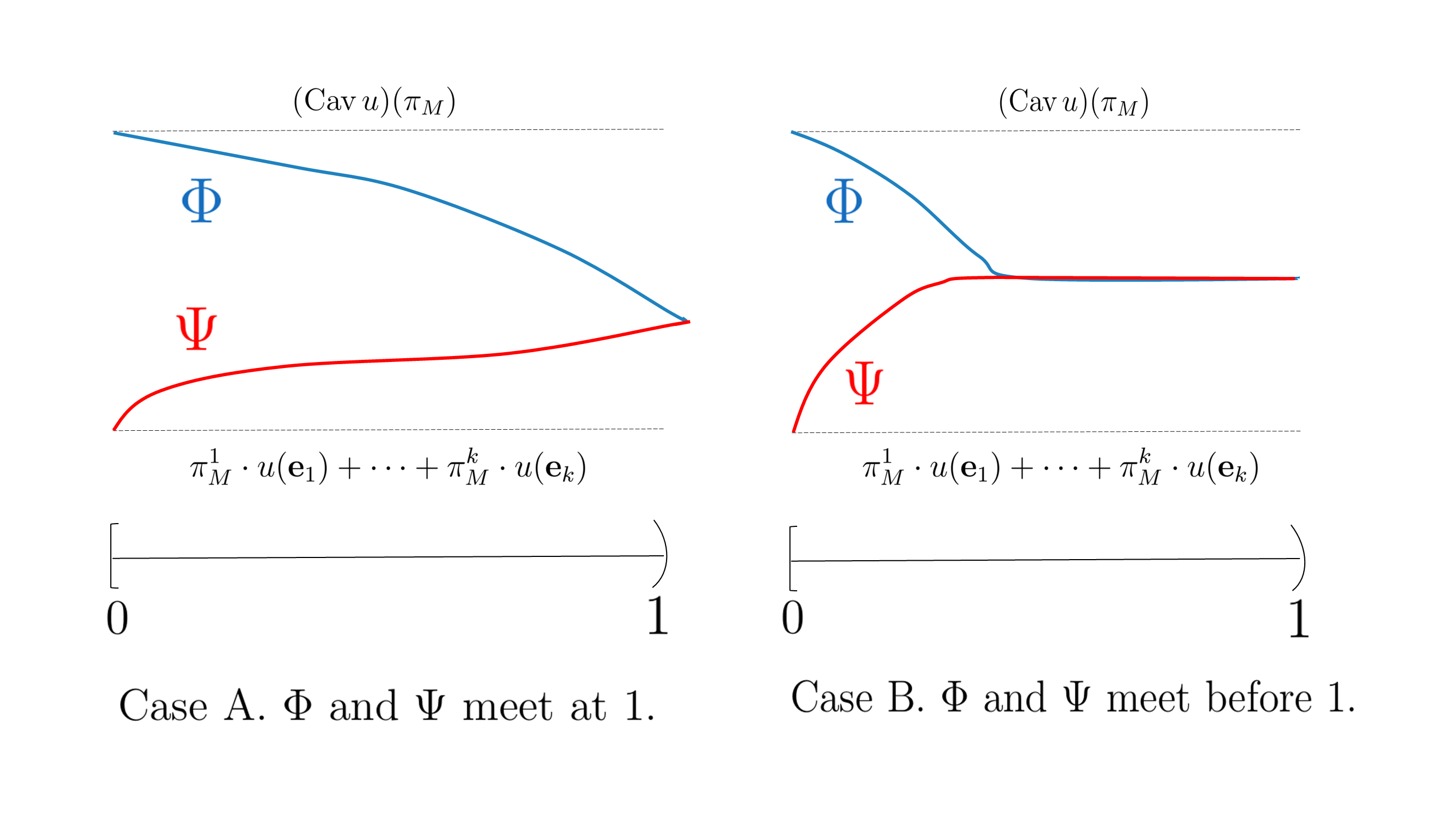}
\includegraphics[scale=0.4]{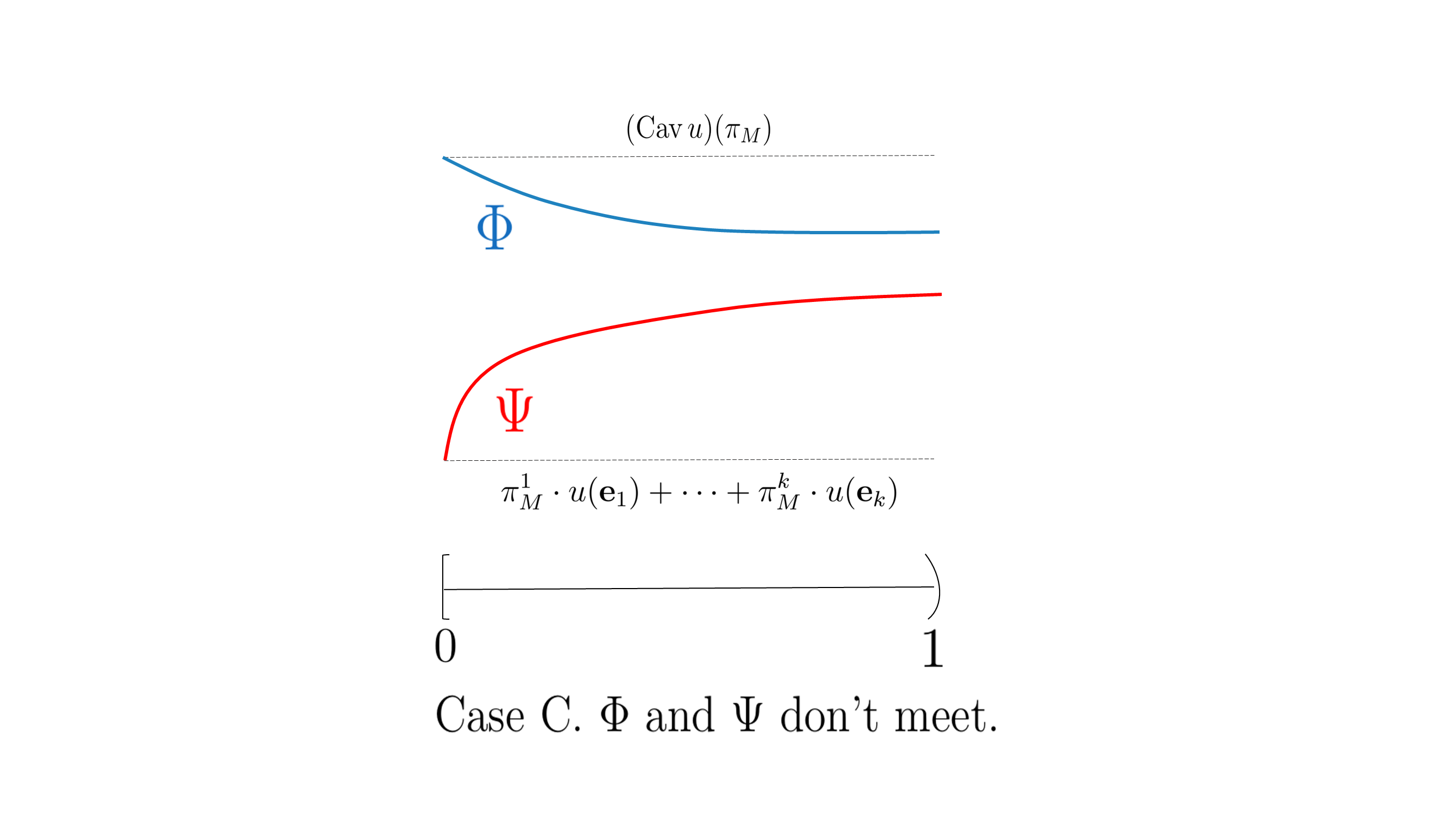}
\caption{The possible relations between the graphs of $\Phi$ and $\Psi$, and the quantities $\cavu (\pi_M)$ and $\pi_M^1 \cdot \cavu (\textbf{m}_1)+\cdots+\pi_M^k \cdot \cavu (\textbf{m}_k)$}
\label{figure 1}
\end{figure}

\subsection{An Upper Bound on the Rate of Convergence to the Asymptotic Value}

Assume that in addition to $M$ being irreducible it is also aperiodic. Put differently, assume that the Markov chain $(X_n)_{n\geq 1}$ is an ergodic one. In Theorem 1 in Lehrer and Shaiderman \cite{MP} it was shown that there exists a number $v_{\infty} \in \R$ such that 
\begin{align}\label{Eq. Uniform Convergence}
    \sup_{\dk} \vert v_{\d}(\cdot) - v_{\infty} \vert \to 0 \quad \text{as} \quad \d \to 1^-. 
\end{align}
That is, the discounted valued $v_{\d}$ converge uniformly on $\dk$ to $v_{\infty}$. The limit $v_{\infty}$ is referred to as the \textit{asymptotic value} of the Markovian persuasion model. 

Our second main result provides a first upper bound on the rate of convergence of $v_{\d}$ to $v_{\infty}$ and reads as follows.

\begin{theorem}\label{Thm3}
 Consider a Markovian persuasion game. Then, for  $1 - (1/2) \cdot \min_{\l \in K} \p^{\l} <\d <1$ it holds that
    \begin{align*}
        \sup_{\dk} \vert v_{\d}(\cdot) - v_{\infty} \vert \leq O\left((1-\d)\log_2 \left(\frac{1}{1-\d}\right)\right). 
    \end{align*}
\end{theorem}

In particular, Theorem \ref{Thm3} implies that the discounted values $v_{\d}$ converge uniformly to $v_{\infty}$ at a rate of at least $(1-\d)^{\alpha}$ for every $\alpha<1$, as $\delta\to 1^{-}$. The initial step for Theorem \ref{Thm3} builds on the following observation:
\begin{align*}
   0\leq  v_{\d}(\p) - v_{\infty} \leq \Phi(\d) - \Psi(\d), \quad \forall \d \in [0,1).
\end{align*}
Indeed, such observation follows from the definitions on $\Phi$ and $\Psi$, Theorems \ref{Thm1} and Theorem \ref{Thm2}, and the identity in \eqref{Eq. Uniform Convergence}. Thereafter, using an appropriate analysis given in Section \ref{Sec Proof of Thm3}, the upper bound in Theorem \ref{Thm3} is enabled by a classic upper bound on the mixing time of ergodic Markov chains due to Levin and Peres \cite{Peres}.

\subsection{Application to Zero-Sum Repeated Games with Incomplete Information on One-Side}\label{Subsec Repeated Game}

As it turns out, the proof techniques developed for the main results in the Markovian persuasion model may be modified to obtain new results to a different model of repeated games. The latter, which was first introduced by Renault \cite{Renault}, sets a dynamic extension of the classic Aumann and Maschler \cite{Aumann} model of repeated games with incomplete information on one-side. 

The model under consideration is described as follows. The set of states $K = \{1,...,k\}$ is associated with a set of zero-sum matrix games, denoted $\{G^{1},...,G^k\}$. The matrices $\{G^{\l}: \l \in K\}$ are assumed to be of equal dimensions. The set of actions of the row player (the maximizer) in those games, is denoted by $I$, and that of the column player (the minimizer) by $J$. The payoff at state $\l \in K$ under the action pair $(i,j) \in I \times J$ equals to $G^{\l}(i,j)$. As in the Markovian persuasion case, we assume that all payoffs are non-negative, i.e., $G^{\l}(i,j)\geq 0$ for every $\l\in K$, $(i,j)\in I \times J$.

The row and column players, called Players 1 and 2, respectively, engage in a repeated game, whose state evolves over time according to a Markov chain $(X_n)_{n\geq 1}$ over $K$ with prior $q \in \dk$ and stochastic transition matrix $M$. The play proceeds by stages; at each stage $n\geq 1$ the following events take place in chronological order:
\begin{itemize}
    \item Player 1 is told the state at stage $n$, being the realized value $x_n$ of $X_n$, unlike Player 2 who only knows $(q,M)$; Player 1 thus learns the zero-sum matrix game $G^{x_n}$ to be played at that stage. 
    \item The players are prescribed to choose simultaneously their $n$'th stage actions $i_n \in I$ and $j_n \in J$.
    \item The pair of actions $(i_n,j_n)$ is publicly announced and becomes common knowledge among the players.
    \item Player 1's payoff for stage $n$ equals to $G^{x_n}(i_n,j_n)$. Note that this payoff is known only to Player 1.
    \item The game proceeds to stage $n+1$.
\end{itemize}

The above description is assumed to be common knowledge among the players. Adopting Renault's terminology \cite{Renault}, the above described game is called a \textit{Markov chain game}. The Aumann and Maschler \cite{Aumann} model is an instance of a Markov chain game where $M = \text{Id}_k$. In such case, the matrix game chosen at the first stage is played repeatedly at every subsequent stage. 

The statement of the optimization problem faced by both players requires the definition of \textit{behavioral strategies}. A \textit{behavioral strategy $\a$ of Player 1} consists of a sequence of mappings $(\a_n)_{n\geq 1}$, where $\a_n : (K \times I \times J)^{n-1}\times K \to \Delta(I)$. That is, under the behavioral strategy $\a$, the $n$'th action $i_n$ of Player 1 is drawn at random from the lottery $\a_n (x_1,i_1,j_1,...,x_{n-1},i_{n-1},j_{n-1},x_n)$. In particular, it may depend on all information available to her at the time. Denote by $\mathcal{A}$ the set of all behavioral strategies of Player 1. Similarly, a \textit{behavioral strategy $\b$ of Player 2} is described by the sequence of mappings $\b_n : (I \times J)^{n-1} \to \Delta(J)$, $n\geq 1$. As is the case for Player 1, Player 2's $n$'th action $j_n$ may take into account his available information at the time: the past pairs of chosen actions $(i_1,j_1,...,i_{n-1},j_{n-1})$. Denote by $\mathcal{B}$ the set of all behavioral strategies of Player 2.

For every pair of behavioral strategies $(\a,\b) \in \mathcal{A}\times \mathcal{B}$ denote by $P^q_{\a,\b}$ the unique probability measure over the space of infinite histories $(K\times I \times J)^{\N}$ along the game.\footnote{We equip $(K\times I \times J)^{\N}$ with the product topology, and thereafter consider the Borel sigma-field on that space. The probability measure $P^q_{\a,\b}$ is defined on events in the Borel sigma-field over $(K\times I \times J)^{\N}$.}  The optimization problem is to find for every discount factor $\d \in [0,1)$ the $\d$-discounted value of the Markov chain game, denoted $V_{\d}(q)$, and defined by:
\begin{align}\label{RG minimax relation}
    V_{\d}(q) & := \max_{\a \in \mathcal{A}}\, \min_{\b \in \mathcal{B}} E^q_{\a,\b}\left[(1-\d)\sum_{n=1}^{\infty} \d^{n-1} \cdot G^{X_n}(i_n, j_n)\right] \nonumber \\
    & = \min_{\b \in \mathcal{B}}\, \max_{\a \in \mathcal{A}} E^q_{\a,\b}\left[(1-\d)\sum_{n=1}^{\infty} \d^{n-1} \cdot G^{X_n}(i_n, j_n)\right],
\end{align}
where the second equality is due to Sion's minimax theorem, and $E^q_{\a,\b}$ stands for the expectation operator with respect to $P^q_{\a,\b}$.

A common underlying assumption when dealing with zero-sum repeated games is that of \textit{pre-play communication}. Such assumption asserts that at every stage $n\geq 1$, prior to choosing their respective actions $(i_n,j_n)$ the players may communicate with each other. This communication is unrestricted, allowing players to send any message they desire. This assumption is not-restrictive in the sense that the value $V_{\d}$ is independent of it. Indeed, the pre-play communication can only add to Player 2's information, and thereby help him minimize Player 1's payoff. As common in the literature, the possibility of pre-play communication is not added to the formal definition of behavioral strategies, and thus we keep the former in above simplified version. 

Lehrer and Shaiderman \cite{MPwSR} have shown that the result of Theorem \ref{Thm2} may also be achieved for Markov chain games. Formally, the statement reads as follows.

\begin{theorem}\label{Thm RG-2}
Consider a Markov chain game. Assume that $M$ is irreducible and let $\p$ be the unique invariant distribution of $M$. Assume that the beliefs $\{\xi_1,...,\xi_k\} \subseteq \Delta(K)$ satisfy
    \begin{itemize}
        \item[\emph{(a)}]  $\xi_1,...,\xi_k$ are affinely independent.
        \item[\emph{(b)}] $M(\Delta(K)) \subseteq conv\, \{\xi_1,...,\xi_k\}$, i.e., for every $\x \in \Delta(K)$, $\x M \in conv\, \{\xi_1,...,\xi_k\}$.
    \end{itemize}
    Then, the trajectory $$\delta \mapsto \gamma_1 V_{\delta}(\chi_1) + ... + \gamma_k V_{\delta}(\chi_k),$$ is non-decreasing on $[0,1)$, where $\gamma_1,...,\gamma_k$ are the unique convex weights satisfying
    \begin{equation*}
        \pi_M = \gamma_1 \chi_1 + ... + \gamma_k  \chi_k.
    \end{equation*}    
\end{theorem}
 
In the current work, we show that the equivalent results of Theorems \ref{Thm1} and \ref{Thm3} also hold for Markov chain games. Formally, our result may be summarized as follows:

\begin{theorem}\label{Thm RG-1,3}
Consider a Markov chain game. We have the following:
    \begin{itemize}
        \item[1.]  For every invariant distribution $\pi$ of $M$, the trajectory $\delta \mapsto V_{\delta}(\pi)$ is non-increasing on $[0,1)$.
        \item[2.] Assume that $M$ is ergodic (i.e., irreducible and aperiodic). Then, there exists a number $V_{\infty} \geq 0$ such that for  $1 - (1/2) \cdot \min_{\l \in K} \p^{\l} <\d <1$ it holds that
    \begin{align*}
        \sup_{\dk} \vert V_{\d}(\cdot) - V_{\infty} \vert \leq O\left((1-\d)\log_2 \left(\frac{1}{1-\d}\right)\right). 
    \end{align*}
    \end{itemize}
\end{theorem}

The first item of Theorem \ref{Thm RG-1,3} can be viewed as an extension of a parallel known result for Aumann and Maschler's model \cite{Aumann}. Such model is a private case of the model considered above where $M = \text{Id}_k$ (i.e., the Markov chain is constant over-time). For such choice of $M$, every belief $\x \in \dk$ is an invariant distribution of $M = \text{Id}_k$. Sorin has shown that in the Aumann and Maschler model $\x \mapsto V_{\delta}(\x)$ is non-increasing on $[0,1)$ for every $\x \in \dk$ (see Proposition 3.11 on p.\ 34 in Sorin \cite{Sorin}). Sorin's proof builds on an algebraic identity, connecting payoffs under any two disjoint discount factors. For the sake of completeness, we present Sorin's proof in Appendix \ref{Appendix B}. Such proof is of particular interest, as it can be directly applied to deduce the first item in Theorem \ref{Thm RG-1,3} as well. 

The second item in Theorem \ref{Thm RG-1,3} provides a new insight on the effect of the ergodicity of $M$ on the rate of convergence of the discounted values $V_{\d}(\cdot)$. In the classic Aumann and Maschler model it was shown that $V_{\d}(\cdot)$ converges to $(\text{Cav}\, U)(\cdot)$ uniformly over $\dk$ at a rate of at least $(1-\d)^{1/2}$, where $U : \dk \to \R$ is the so called `value of the one shot non-revealing game', defined by:
\begin{align*}
    U(\x) := \text{val} \left( \sum_{\l \in K} \x^{\l} G^{\l} \right), \quad  \forall \x \in \dk.
\end{align*}
However, by the second item of Theorem \ref{Thm RG-1,3} we have that under the ergodicity assumption $V_{\d}(\cdot)$ converges uniformly on $\dk$ at a rate of at least $(1-\d)^{\a}$ for every $0<\a < 1$. Such a result suggests that the mixing properties of the stochastic transition matrix $M$ may affect the rate of convergence of the discounted valuations. The tightness of the above mentioned bounds on the rate of convergence remains an open problem. It is our hope that the techniques and tools of the current research may advance the work on this problem. 

The proof of Theorem \ref{Thm RG-1,3} which appears in Section \ref{Sec RG proof}, is achieved by surveying the required modifications to the proofs of the main results in regards to the Markovian persuasion model. Before proceeding to such proofs, we move to the development of a suitable toolbox for the Markovian persuasion model. Such toolbox will be necessary for all subsequent proofs in the current work.

\section{The Toolbox}\label{Sec Toolbox}

\subsection{A Markov Decision Problem Reformulation}\label{Subsec MDP reformulation}

As first noted by Ely \cite{Ely} and Renault, Solan, and Vieille \cite{Solan}, the Markovian persuasion model may be reformulated in terms of a Markov decision process (MDP) having the following attributes:
\begin{itemize}
    \item The set of \textit{states} is $\dk$. The initial state is denoted $q$. 
    \item At each state $\xi \in \dk$ the \textit{set of available actions at} $\xi$, denoted $\mathcal{A}_{\xi}$, is defined by 
    \begin{align*}
        \mathcal{A}_{\xi} = \Bigg\{ (\a_s,\xi_s)_{s\in S}\,:\, \a_s\geq0,\, \x_s \in \dk, ~ \forall s\in S, ~ \sum_{s \in S} \a_s =1, ~ \sum_{s\in S} \a_s \xi_s = \xi   \Bigg\}.
    \end{align*}
    In words, $\mathcal{A}_{\xi}$ describes all discrete distributions over $|S|$ beliefs in $\dk$ whose mean equals $\xi$, also known as the set of \textit{splits} of $\xi$ to $|S|$ beliefs in $\dk$.
    \item The \textit{transition rule} $\rho$ associates with each state $\xi \in \dk$ and action $\chi = (\a_s,\xi_s)_{s\in S} \in \mathcal{A}_{\xi}$ a new state, denoted  $\rho (\xi,\chi)$, picked stochastically by the law
    \begin{align*}
        \rho (\xi,\chi) = \xi_sM~~ \text{with probability}~~ \a_s.
    \end{align*}
    \item The \textit{reward function} $r$ assigns to each state $\xi$ and action $\chi = (\a_s,\xi_s)_{s\in S} \in \mathcal{A}_{\xi}$, the number $r(\xi,\chi)$ defined by
    \begin{align*}
        r(\xi,\chi) := \sum_{s \in S} \a_s \cdot u(\xi_s).
    \end{align*}
\end{itemize}

For the sake of completeness, let us briefly discuss the validity of the above reformulation, albeit that it was previously addressed in detail in works such as \cite{Ely},\cite{Solan}, and \cite{MP}. Such validity is rooted in the Aumann and Maschler Splitting Lemma (e.g., Aumann and Maschler \cite{Aumann}). Consider a strategy $\s = (\s_n)_{n\geq 1} \in \Sigma$, and note that for each $n\geq 1$, $q_n$'s distribution depends on $\s_1,...,\s_{n-1}$, but not on $\s_n,\s_{n+1},...$~. Such Lemma implies that given that $q_n=\x$, for any $\chi \in \mathcal{A}_{\x}$ the sender can construct the lottery $\s_n$ so that $p_n$ would be distributed according to $\chi$.

On the other hand, as relation \eqref{Martingale} implies that the conditional distribution of $p_n$ given $q_n=\x$ lies in $\mathcal{A}_{\x}$, we have that at each stage $n$ the sender should choose which `split' of $q_n$ to employ. Therefore, in essence, the prior $q_n$ corresponds to the state of MDP at stage $n$, whereas the conditional distribution of the posterior $p_n$ given $q_n$, corresponds to the $n$'th action in the MDP. 

Under such a correspondence, the definition of the transition rule $\rho$ of the MDP follows immediately from relation \eqref{M-shift}, whereas the reward function $r$ was defined so that the payoff at stage $n$ of the MDP, corresponds to the conditional expectation of $u(p_n)$ given $q_n$.

Abusing notation, a \textit{pure strategy} $\s$ of the sender in the MDP is thus described by a sequence of stage strategies $\s_n$, so that $\s_n: (\x_1,\chi_1,...\x_{n-1},\chi_{n-1},\x_n) \mapsto \chi_n \in \A_{\x_n}$, where $\x_m$ and $\chi_m$ denote the $m$'th state and action, respectively, along the play of the MDP. Thus $\Sigma$ described the set of all pure strategies $\s$ in the MDP. A \textit{mixed strategy} in the MDP is will be defined to be a finite distribution over $\Sigma$, i.e., a random choice of a strategy among a finite set of pure strategies. 

\subsection{Recursive Formula and Optimal Strategies}\label{SubSec Opt Str}

The MDP reformulation above gives rise for a recursive formula for  $v_{\d}(\cdot)$. By applying the dynamic programming principle (e.g., Theorem 1.22 in \cite{SolanB}) we obtain that for any $\xi \in \dk$ it holds\footnote{Note that in \eqref{Belmann}, we have restricted ourselves to maximization over pure strategies rather than mixed ones. This is without a loss of generality due to linearity arguments.}
\begin{align}\label{Belmann}
v_{\delta}(\xi)  = \sup_{\chi = \{(\a_s,\x_s)\} \in \mathcal{A_{\xi}}} \bigg\{ (1-\delta)\sum_{s \in S} \a_s \cdot u (\x_s) + \delta \sum_{s \in S} \a_s \cdot v_{\delta}(\x_s M) \bigg\}.
\end{align}
Let us denote by $(v_{\d} \circ M) : \dk \to \R_+$ the mapping defined by $(v_{\d}\circ M)(q) := v_{\d}(qM)$. Then, we may rewrite Eq.\ \eqref{Belmann} as follows:
\begin{align}\label{Belmann2}
v_{\delta}(\xi)  = \sup_{\chi = \{(\a_s,\x_s)\} \in \mathcal{A_{\xi}}} \sum_{s \in S} \a_s \cdot \bigg\{ (1-\d)\cdot u + \d \cdot ( v_{\d} \circ M )  \bigg\} (\x_s).
\end{align}

\subsubsection{On the Structure of Optimal Strategies}\label{Sub-sub sec Functional Equation}

Carath\'{e}odory's Theorem (see, e.g., Corollary 17.1.5 in \cite{Rock}) implies that for any signal set $S$ with $|S| \geq k$ the right hand-side in Eq.\ \eqref{Belmann2} equals to $(\text{Cav} f_{\d})$, where $f_{\d}: \dk \to \R_+$ is defined by $f_{\d}(\xi) :=  (1-\d)\cdot u (\x) + \d \cdot ( v_{\d} \circ M)(\x)$ for $\x \in \dk$.\footnote{In particular for $\d = 0$, we have that $(\text{Cav} f_{\d}) = \cavu$, which in turn with Eq.\ \eqref{Belmann2} implies that $v_{0}(\xi) = \cavu(\xi)$, recovering the result of Gentzkow and Kamenica \cite{Kamenica}.} In conjunction with \eqref{Belmann2} this implies that $v_{\delta} = (\text{Cav} f_{\d})$. Hence, $v_{\d}$ is concave on $\dk$, and thus so is $(v_{\d} \circ M)$. As any concave function on $\dk$ must be continuous, the upper semi-continuity of $u$ implies that $f_{\d}$ is upper semi-continuous as well.

The upper semi-continuity of $f_{\d}$ is of particular importance, as together with Carath\'{e}odory's Theorem (see, e.g., Corollary 17.1.5 in \cite{Rock}) it implies that for any $\xi \in \dk$ there exists an optimal split $\chi_{\d}^*(\xi) \in \mathcal{A_{\xi}}$  such that the supremum on the right hand-side of Eq.\ \eqref{Belmann2} is achieved at $\chi = \chi_{\d}^*(\xi)$. Moreover, by Carath\'{e}odory's Theorem we may take $\chi_{\d}^*(\xi)$ to supported on at most $k$ beliefs. Therefore, an optimal strategy in the Markovian persuasion model $\s_{\d}^* \in \Sigma$ can be taken to follow the stationary strategy (in the MDP) $\chi_{\d}^* :\dk \to \bigcup_{\x \in \dk} \mathcal{A}_{\xi}$ that assigns to any belief $\xi \in \dk$ the split $\chi_{\d}^*(\xi)$. That is, $\s_{\d}^* \in \Sigma$ is defined such that at every stage $n\geq 1$, if $q_n =\xi$, then $\s_{\d}^*$ splits $q_n$ to $\chi_{\d}^*(\xi)$. 

\subsubsection{The Signal's Cardinality Invariance Principle}\label{SICP property}
By the above discussion and structure of $\s_{\d}^*$ it follows that $v_{\d}(\cdot)$ does not depend on $S$ as long as $|S|\geq k$. In particular, the sender does not gain from additional signals beyond $k$ signals. Also, from here on forward we assume without loss of generality that $S = \N$.

\subsection{$M$-Martingale Property}\label{Subsec Optimal Distributions}
By combining the results of Eqs.\ \eqref{Martingale} and \eqref{M-shift} we obtain that for any signaling policy $\s \in \Sigma$ and prior $q \in \dk$ the corresponding posteriors sequence $(p_n)_{n\geq 1}$ satisfies the so called `$M$-martingale property':
\begin{align*}
    E_{q,\s} (p_{n+1}\,|\, p_n) = p_n M, \quad \forall n\geq 1.
\end{align*}
In particular, using an iterative argument, at the basis of which we have $E_{q,\s}\, p_1 = q$, we obtain that:
\begin{align}\label{Mean of M-Mart.}
    E_{q,\s}\, p_n = qM^{n-1}, \quad \forall \s \in \Sigma,~ \forall n\geq 1.
\end{align}
This relation is of particular interest for $q =\pi$, where $\pi$ is any invariant distribution of $M$. Indeed, it implies that the posteriors sequence $(p_n)_{n\geq 1}$ has fixed mean: $E_{\pi,\s} p_n = \pi$, for every $\s \in \Sigma$ and every $n\geq 1$.

\headings{The distributions of optimal posteriors.} For every $q \in \dk$ and $n\geq 1$, let us denote by $\mu^*_n (\d,q) \in \Delta(\dk)$ the distribution of the $n$'th posterior associated with the prior $q$ and strategy $\s_\d^*$. That is, 
\begin{align*}
  \mu^*_n (\d,q) (E) := P_{q,\s_\d^*}\, (p_n \in E ), \quad  \forall ~ \text{Borel}~ E \subseteq \dk,
\end{align*}
where as before $p^{\l}_n = P_{q,\s_\d^*}\,(X_n = \l\,|\,s_1,...,s_n)$, $\l \in K$. Two important properties of $\mu^*_n (\d,q)$ are the following:
\begin{itemize}
    \item By the definition of $\s_\d^*$, $\mu^*_n (\d,q)$ is supported on at most $k^n$ beliefs.
    \item The mean of the distribution $\mu^*_n (\d,q)$ equals to $qM^{n-1}$, as follows from \eqref{Mean of M-Mart.}.
\end{itemize}
The above two properties may be summarized as follows. For each $q \in \dk$ and $n\geq 1$, $\mu^*_n (\d,q)$ is a split of $qM^{n-1}$ to at most $k^n$ beliefs. That is, $\mu^*_n (\d,q) \in \mathcal{A}_{qM^{n-1}}$.

\subsection{The Random Duration Payoff}\label{Subsec Random Duration}

Consider a scenario in which the number of stages in the Markovian persuasion model is random. That is, the game between the sender and the receiver lasts only a (random) finite number of stages. Furthermore, under such a scenario, let us assume that the sender evaluates his total payoff as the sum of stage payoffs along the game. If we model the random duration by a geometric random variable $Y$ with parameter $x \in (0,1]$, i.e.,
\begin{align*}
    P(Y=n) = (1-x)^{n-1}\cdot x, \quad \forall n\geq 1,
\end{align*}
then the sender's \textit{expected random duration payoff} associated with the signaling policy $\s \in \Sigma$, denoted $\gamma^{\text{r-d}}_x(q,\s)$, is given by the formula
\begin{equation}\label{ran dur pay}
\gamma^{\text{r-d}}_x(q,\s) = E_{q,\s} \left(\sum_{n=1}^{\infty} u(p_n) \cdot \textbf{1}\{n \leq Y\}\right).
\end{equation}
In addition let us assume that $Y$ is independent of all other events along the game. Denote by $v^{\text{r-d}}_x(q):= \sup_{\s \in \Sigma} \gamma^{\text{r-d}}_x(q,\s)$ the `random duration' value. 

The following claim uncovers the relations between the discounted values and random duration values.
\begin{claim}\label{Claim Random Duration}
For every $q \in \dk$, and $x \in (0,1]$, we have
\begin{equation}\label{Eq. disint}
v^{\emph{r-d}}_x(q) = \frac{v_{1-x}(q)}{x}.
\end{equation}
Moreover, the strategy $\s_{1-x}^* \in \Sigma$ defined in Subsection \ref{SubSec Opt Str} is optimal under the expected random duration payoff, that is, $\gamma^{\emph{r-d}}_x(q,\s_{1-x}^*) = v^{\emph{r-d}}_x(q)$.
\end{claim}
\begin{proof}[Proof of Claim \ref{Claim Random Duration}.]
Since $Y$ is independent of $(X_n)$ and $(p_n)$, we may disintegrate it from (\ref{ran dur pay}), and then change the order of summation to obtain that for every $\s \in \Sigma$ we have that
\begin{equation*}
\begin{split}
\gamma^{\text{r-d}}_x(q,\s) & = \sum_{N=1}^{\infty} (1-x)^{N-1} \cdot  x \cdot E_{q, \s} \left(\sum_{n=1}^{N} u(p_n)\right)\\
& = \sum_{n=1}^{\infty} \left(\sum_{N = n}^{\infty} (1-x)^{N-1}\cdot x \right)  E_{q,\s}\, u(p_n)\\
& = \frac{1}{x}  \sum_{n=1}^{\infty} x\cdot (1-x)^{n-1}\,\, E_{q,\s} \, u(p_n) = \frac{1}{x}\cdot  \gamma_{1-x}(q,\s).
\end{split}
\end{equation*}
Maximizing over $\s \in \Sigma$ we obtain Eq.\ (\ref{Eq. disint}) holds. This, together with the definition of $\s_{1-x}^*$ implies the optimality of $\s_{1-x}^*$ under the expected random duration payoff, thus completing the proof of the claim. 
\end{proof}

\section{Proof of Theorem \ref{Thm1}}\label{Sec-Proof of Thm 1}

\subsection{A Competitive Stochastic Game with Signals}\label{Subsec Stochastic Game with Signals}

For the sake of the proof of Theorem 
\ref{Thm1} let us concentrate from this point onward on a fixed invariant distribution $\pi$ of $M$. Our starting point for the proof is the introduction of a zero-sum repeated game between the sender and an adversary, which we denote by $\Gamma_x$, where $x$ is a fixed number in $(0,1)$.

\subsubsection{The Primitives of $\G$}

The game $\G$ fits the model of a zero-sum stochastic game with signals (e.g.\ Subsections 1.a and 1.e in Chapter 4 in Mertens, Sorin, and Zamir \cite{MSZ} or Chapter 4.3 in Solan and Ziliotto \cite{SolanZiliotto}). The primitives defining $\G$ are the following:

\begin{itemize}
    \item \textit{States.} The states of $\Gamma_x$ are elements of $\O :=\dk \cup \{\pi^* , \pi^{**}\}$, where $\pi^*$ and $\pi^{**}$ are two distinguished copies of $\pi$. 
    \item \textit{Actions.} At each state $\xi \in \dk$, the available actions of the sender are the splits of $\xi$:
     \begin{align*}
        \mathcal{A}_{\xi} = \Bigg\{ (\a_i,\xi_i)_{i\geq 1}\,:\, \a_i\geq0,\, \x_i \in \dk, ~ \forall i\geq 1, ~ \sum_{i=1}^{\infty} \a_i =1, ~ \sum_{i=1}^{\infty} \a_i \xi_i = \xi   \Bigg\}.
    \end{align*}
    where we recall that the signal set $S=\N$ (see concluding discussion of Subsection \ref{SubSec Opt Str}). For $\xi \in \{\pi^* , \pi^{**}\}$, the available actions are defined to be $\mathcal{A}_{\pi}$.

    As for the adversary, the set of available actions at each state $\xi \in \O$, is  $C = \{0,1\}$. That is, the state of the game has no effect on the actions available to the adversary.

    \item \textit{Payoff function.} The adversary's payoff to the sender is described by a function $f:\O \times \bigcup_{\x \in \O} \A_{\x} \times C \to \R$. The payoff associated with state $\xi \in \O$ and actions $\chi = \{(\a_i,\x_i)\} \in \mathcal{A}_{\x}$, $c \in \{0,1\}$, is given by
    \begin{align*}
      f(\x,\chi,c) := \sum_{i=1}^{\infty} \a_i \cdot u(\x_i).  
    \end{align*}
    In particular, the payoff that the adversary transfers to the sender depends only on the split of $\x$ that the sender chooses to employ.

    \item \textit{Transition rule.} The transition rule is described by a lottery $t:\O\times \bigcup_{\x \in \O} \A_{\x} \times C \to \Delta(\O \times \Lambda_1 \times \Lambda_2 )$, where $\Lambda_1$ is termed the \textit{signal set of the sender} and $\Lambda_2$ the \textit{signal set of the adversary}.\footnote{The set of signals $\Lambda_1$ corresponds to signals that the sender obtains along $\G$, as opposed to the ones she transmits (chosen from the set $S=\N$) to induce splits in $\A_{\x}$, $\x \in \O$. This would be made clear in the following subsection.}  Such a lottery is defined by,
    \begin{align*}\label{transition rule}
t(\x,\chi,0) & = \left\{
       \begin{array}{ll}
        (\pi^*,\pi^*,\emptyset) , & \hbox{with prob.\,\,}\,\, x\\ 
        (\xi_i M, \xi_i M, \emptyset) , & \hbox{with prob.\,\,}\,\, (1-x)\cdot \a_i
       \end{array}
     \right.
     \\ 
     \text{and}
     \\
t(\x,\chi,1) & = \left\{
       \begin{array}{ll}
        (\pi^*,\pi^*,\emptyset) , & \hbox{with prob.\,\,}\,\, x\\ 
        (\pi^{**},\pi^{**},\emptyset) , & \hbox{with prob.\,\,}\,\, (1-x)
       \end{array}
     \right.,
\end{align*}
for every $\x \in \O$ and every $\chi = \{(\a_i,\x_i)\} \in \mathcal{A}_{\x}$. In particular, we may take $\Lambda_1 = \O$ and $\Lambda_2 = \emptyset$.
\end{itemize}

\subsubsection{The Play Description of $\G$}\label{Play of G_x}

The game $\G$ is played by stages. For each stage $n\geq 1$ denote by $\o_n$ the  state at the $n$'th stage of $\G$.
The game starts from state $\o_1 = \pi$, which is common knowledge among the players. At the first stage  the sender and the adversary are prescribed to choose simultaneously actions $\chi_1 \in \A_{\pi}$ and $c_1 \in C$.

The game $\G$ then proceeds iteratively as follows. At any stage $n\geq 2$, given $(\o_{n-1},\chi_{n-1},c_{n-1})$ a tuple $(\zeta_n,\lambda_{n-1}^1,\lambda_{n-1}
^2) \in \O \times \Lambda_1 \times \Lambda_2$ is chosen according to the lottery $t(\o_{n-1},\chi_{n-1},c_{n-1})$, i.e., $\lambda_{n-1}^1=\zeta_n$ and $\lambda_{n-1}
^2 = \emptyset$. Upon this choice, the play at stage $n$ proceeds in the following chronological order:
\begin{itemize}
    \item The new state $\o_n$ is set to equal $\zeta_n$.
    \item The sender observes the signal $\lambda_{n-1}^1 = \o_n$, whereas the adversary observes the signal $\lambda_{n-1}^2 = \emptyset$. 
    \item The sender and the adversary are prescribed to choose simultaneously their respective actions $\chi_{n} \in \A_{\o_{n}}$ and $c_{n} \in C$.
   \end{itemize}
Given $(\o_n, \chi_n, c_n)$ the game $\G$ then proceeds stage $n+1$.

Let us now give some important remarks. By the above description, the sender observes at each stage $n$ the current state $\o_n$, unlike the adversary who obtains no information regarding $(\o_n)_{n\geq 2}$. Moreover, the adversary obtains no information regarding the sender's actions. In contrast, if at some stage $n\geq 2$,  $\o_n = \pi^{**}$, the sender knows that $c_{n-1}=1$, thus getting some information regarding the adversary's past actions. However, it may be the case that $c_{n-1}=1$ and $\o_n = \pi^*$, in which case the sender does not know the choice of $c_{n-1}$.  

We assume that the description of the play and primitives of $\G$ is commonly known among the players. In addition, we assume that the players have perfect recall of the signals they observe, as well as of their own past played actions. 

\subsubsection{The Motivation Behind $\G$}\label{Motivation to G_x}

Broadly speaking, the game $\G$ can be viewed as a generalization of the Markovian persuasion model, where the sender has no full control of the receiver's evolving beliefs. If we view the states $(\o_n)_{n\geq 1}$ in $\G$ as a generalized version of the prior beliefs $(q_n)_{n\geq1}$, then at the start of each stage of $\G$ chance may erase the receiver's memory at random according to a Bernoulli process with parameter $x$. Upon a memory erasure at stage $n$, the stationarity of $(X_n)_{n\geq 1}$ implies that $q_n$ should return to the invariant distribution $\pi$. For this reason, we model memory erasures of chance in $\G$ via transitions to state $\pi^*$.

The adversary in $\G$ can be thought of as a competitive strategic device that is able to erase the receiver's memory as well. As follows from the definition of $\G$, the adversary's memory erasures affect the game only in the case where chance didn't make a move, i.e., performed a memory erasure. The adversary's erasures are made distinguishable from those of chance, by introducing the state $\pi^{**}$. Indeed, the visits to $\pi^{**}$ may occur only when the adversary chooses to erase the memory of the receiver, i.e., take the action $c=1$. It is worth noting that the adversary in $\G$ plays in the `dark', in the sense that he has recall of his own actions only, unlike the sender who obtains information both on the moves of chance and the adversary.

From a payoff perspective, $\G$ is similar to the Markovian persuasion model in that the sender's stage payoffs are fully at hers control. Indeed, the payoff function $f$ only depends on the splits of the sender, similarly to the stage payoffs $\{u(p_n)\,:\, n\geq 1\}$ in the Markovian persuasion model which are determined by the posterior beliefs $(p_n)_{n\geq 1}$, being splits of $(q_n)_{n\geq 1}$. In that view, compared to the Markovian persuasion model, in $\G$ the sender also has full control of the receiver's posterior beliefs, but by contrast, has only partial control of the receiver's prior beliefs, which are affected by the moves of chance and the adversary.

\subsubsection{Strategies and $N$-Ces\`aro Payoffs}\label{Strategies in G_x}

As follows from the definition and play description of $\G$, prior to choosing hers $n$'th action $\chi_n$, the sender knows $\o_1,\chi_1,...,\o_{n-1},\chi_{n-1}$ together with today's state $\o_n$. Therefore, a \textit{strategy} $\theta$ of the sender in $\G$, is defined by a sequence of stage strategies $(\theta_n)_{n\geq 1}$, where $\theta_n : (\o_1,\chi_1,...,\o_{n-1},\chi_{n-1},\o_n) \mapsto \A_{\o_n}$. Denote by $\Theta$ the set of all strategies of the sender. 

As it pertains to the adversary, the sequence $(c_1,...,c_{n-1})$ and initial state $\o_1 =\pi$ constitute his available information when coming to choose his $n$'th action $c_n$. As $\o_1=\pi$ is fixed, we define a \textit{strategy} $\tau$ for the adversary in $\G$, in terms of a sequence of stage strategies $(\tau_n)_{n\geq 1}$, where $\t_n : (c_1,...,c_{n-1}) \mapsto \Delta(C)$. Denote by $\T$ the set of all strategies of the adversary. 

Note that a strategy of the adversary allows him to \textit{randomize} between actions, whereas that of the sender attaches to any sequence $(\o_1,\chi_1,...,\o_{n-1},\chi_{n-1},\o_n)$ a \textit{deterministic} action. Thus, there is an asymmetry from that point of view. However, for the sake of the proof, such definitions of strategies simplify matters, and help avoid unnecessary technical complications.

After $N$ stages of play of $\G$, the \textit{history of play} is given by the sequence $(\o_1,\chi_1,c_1,..., \o_N, \chi_N, c_N)$. By the description of $\G$, the former sequence includes all information about the first $N$ stages of $\G$. Therefore, the possible histories of play of length $N$ are elements in $H_N :=(\O \times \bigcup_{\x} \A_{\x}\times C)^N$. Let $H := (\O \times \bigcup_{\x} \A_{\x}\times C)^{\N}$, and equip $H$ with the sigma-field $\mathcal{H}$ generated by all subsets of the form $h \times (\O \times \bigcup_{\x} \A_{\x}\times C)^{\N}$, where $h \in H_N$ for some $N$. By Kolmogorov's Extension Theorem each pair of strategies $(\theta,\t) \in \Theta \times \T$ induces a unique probability measure on $(H, \mathcal{H})$, denoted $P^x_{\theta,\t}$, consistent with finite histories of plays, i.e., 
\begin{align*}
    P^x_{\theta,\t}(\o_1,\chi_1,c_1,..., \o_N, \chi_N, c_N) & = \textbf{1}\{\theta_1(\o_1) =\chi_1\} \times \t_1 [c_1] \times \prod_{n=2}^N t(\o_{n-1}, \chi_{n-1},c_n) [\o_n] \\
    & \quad \times \textbf{1}\{ \theta_n(\o_1,\chi_1,...,\o_{n-1},\chi_{n-1},\o_n)=\chi_n \} \\
    & \quad \times \t_n (c_1,...,c_{n-1})[c_n],
\end{align*}
for every history of play $(\o_1,\chi_1,c_1,..., \o_N, \chi_N, c_N) \in H_N$.
\bigskip 

For every $N\geq 1$, the $N$-Ces\`aro expected payoff associated with a pair of strategies $(\theta , \tau) \in \Theta \times \T$, is denoted by $\gamma_N^x (\theta,\t)$, and defined by 
\begin{align*}
   \gamma_N^x (\theta,\t):= E^x_{\theta,\t} \left( \frac{1}{N} \sum_{n=1}^N f(\o_n,\chi_n,c_n)\right).
\end{align*}

\subsection{The Plan of the Proof}\label{Plan of the Proof}

The proof will build on two separate propositions. For the statement of such propositions we introduce the notation $o(1)$  to denote a sequence of numbers $(a_N)_{N\geq 1}$ satisfying $a_N \to 0$ as $N \to \infty$. The propositions read as follows. 

\begin{proposition}\label{prop1}
There exists a strategy $\theta^{*} \in \Theta$  such that for every strategy $\t \in \T$ it holds that
\begin{equation*}
 \gamma_N^x (\theta^*,\t) \geq v_{1-x}(\pi) - o(1). 
\end{equation*}
\end{proposition}

\begin{proposition}\label{prop2}
For every $y \in (x,1]$ there exists a strategy $\t^y \in \T$ such that for every strategy $\theta \in \Theta$ it holds that
\begin{equation*}
 \gamma_N^x (\theta,\t^y) \leq v_{1-y}(\pi) + o(1). 
\end{equation*}
\end{proposition}

Proposition \ref{prop1} and \ref{prop2} imply that for any $0 < x <y \leq 1$ it holds that
\begin{align*}
v_{1-x}(\pi) - o(1) & \leq \sup_{\theta \in \Theta} \inf_{\t \in \T}  \gamma_N^x (\theta,\t)\\
& \leq \inf_{\t \in \T} \sup_{\theta \in \Theta} \,  \gamma_N^x (\theta,\t) \leq v_{1-y}(\pi) + o(1).
\end{align*}
By letting $N \to \infty$ we obtain that $v_{1-x}(\pi) \leq v_{1-y}(\pi)$ for $0 < x <y \leq 1$. Taking $\d = 1-x$, we obtain that $\d \mapsto v_{\d}(\pi)$ is decreasing on $[0,1)$, as desired.

\subsection{Probabilistic framework for the proofs of Propositions \ref{prop1} and \ref{prop2}}\label{Subsec Probabilistic Framework}

Important preparations required for the proofs of Propositions \ref{prop1} and \ref{prop2} consist of setting up a probabilistic framework in $\G$. We define the sequence of variables $(\k_n)_{n \geq 0}$ iteratively by
\begin{align*}
\k_0 & := 1, \nonumber \\
\k_n & := \inf\,\{\, n \geq 1\,:\, \o_{\k_0 + \k_1+ \cdots + \k_{n-1}+n} = \pi^*\, \}, \quad  n \geq 1.
\end{align*}
By the definition of $\G$ we have that for any pair of strategies $(\theta , \tau) \in \Theta \times \T$, the sequence of random variables $(\k_n)_{n\geq 1}$ defined on $(H, \mathcal{H}, P^x_{\theta,\t})$ are i.i.d.\ geometrically with parameter $x$. In words, the sequence $(\k_n)_{n\geq 1}$ counts the number of stages between successive memory erasures performed by chance in $\Gamma_x$.

Next, define for any $N \in \N$ the random integer $T_N$ by
\begin{equation*}
T_N := \max\, \{\, n \in \{0,1,..., N-1\} \, : \, \k_0 + \k_1+ \cdots + \k_n \leq N\, \}.
\end{equation*}
In words, $T_N$ stands for the number of memory erasures performed by chance at the end of the $N$'th stage of $\G$. Note that $T_N$ also admits a representation in terms of a sum of $N-1$ Bernoulli random variables given by
\begin{align}\label{T_N Bernoulli Representation}
    T_N = \textbf{1}\{\o_2 = \pi^*\} + \cdots + \textbf{1}\{\o_{N} = \pi^*\}.
\end{align}

Lastly, for each $N \in \N$, set $\l_N := \k_0+\k_1+\cdots + \k_{T_N}$. The random variable $\l_N$ describes the total number of stages from the start of the game $\G$ up to the final memory erasure performed by chance up to stage $N$. In particular $\l_N \leq N$.

The following decomposition of $\gamma_N^x$ in terms of the probabilistic framework just defined will be useful for the proofs of Propositions \ref{prop1} and \ref{prop2}. For any $N \geq 1$ and $(\theta , \tau) \in \Theta \times \T$ we have

\begin{align}\label{Eq.1}
\gamma_N^x (\theta,\t) & = E^x_{\theta,\t} \left(\frac{1}{N}\sum_{n=1}^{\l_N} f(\o_n,\chi_n,c_n) \right) \quad \quad \quad := A_N (x,\theta,\t)  \\
& \quad \quad \quad + E^x_{\theta,\t} \left(\frac{1}{N}\sum_{n= \l_N +1}^{N} f(\o_n,\chi_n,c_n) \right), \quad \quad \quad := B_N (x,\theta,\t) \nonumber
\end{align}
where we use the convention that $\sum_{n = L+1}^L f(\o_n,\chi_n,c_n)=0$ for any $L \in \N$.

Let us now analyze $B_N (x,\theta,\t)$. By the non-negativity of $u$ and the definition of $f$ we deduce that 
\begin{align}\label{Eq.3}
0 \leq  B_N (x,\theta,\t)  \leq E^x_{\theta,\t} \left( \frac{(N - \l_N)}{N}\, \Vert u \Vert_{\infty} \right).
\end{align}
for any pair of strategies $(\theta , \tau) \in \Theta \times \T$. We need to estimate $E^x_{\theta,\t} \left(N -\l_N\right)$. Let us claim that for every $1 \leq n \leq N$ we have
\begin{equation*}
\{N - \l_N = n\} = \{\o_{N-n}=\pi^*, \o_{t}\neq \pi^*, \,\,\, \forall t \in \{N-n+1,...,N\}\}.
\end{equation*}
Indeed, as $\l_N$ is the last stage among the first $N$ stages of $\G$ at which chance performed a memory erasure, the event $\{N - \l_N = n\}$ states that the last such erasure took place at stage $N-n$.  Therefore, we have for every $(\theta , \tau) \in \Theta \times \T$ that
\begin{align}\label{C analysis}
E^x_{\theta,\t} \left(N -\l_N\right) &= \sum_{n=0}^N n \, P^x_{\theta,\t} \left(\{N - \l_N = n\} \right) \nonumber \\
& = \sum_{n=1}^N n\, P^x_{\theta,\t} \left(\{\o_{N-n}=\pi^*, \o_{t}\neq \pi^*, \,\,\, \forall t \in \{N-n+1,...,N\}\} \right)\\
& = \sum_{n=1}^N n
\, x(1-x)^{n} \leq (1-x) \sum_{n=1}^{\infty} n (1-x)^{n-1}x = \frac{1-x}{x}. \nonumber
\end{align}
Hence, Eq.\ (\ref{C analysis}) in conjunction with Eq.\ (\ref{Eq.3}) implies that for every $(\theta , \tau) \in \Theta \times \T$, and every $N \geq 1$, it holds
\begin{equation}\label{key bound on C}
0 \leq  B_N (x,\theta,\t)  \leq \frac{1-x}{x} \cdot \frac{\Vert u \Vert_{\infty}}{N}.
\end{equation}
Hence, by Eqs.\ (\ref{Eq.1}) and (\ref{key bound on C}) we obtain that for any $N \geq 1$ and $(\theta , \tau) \in \Theta \times \T$ it holds that
\begin{equation}\label{Eq. approx.}
\vert \gamma_N^x (\theta,\t) - A_N (x,\theta,\t) \vert \leq \frac{1-x}{x} \cdot \frac{\Vert u \Vert_{\infty}}{N}.
\end{equation}
With Eq.\ (\ref{Eq. approx.}) in mind, we proceed with the proofs of Propositions \ref{prop1} and \ref{prop2}.

\subsection{Proof of Proposition \ref{prop1}}\label{Subsec Prop 1}

Prior to defining $\theta^* \in \Theta$, we introduce for each integer $m \geq 0$ the notation $\k_0 \oplus \k_m := \sum_{i \leq m} \k_i$. The strategy $\theta^*$ is defined by the following rules. For each integer $m \geq 0$ and integer $0\leq i < \k_{m+1}$ play at stage $\k_0 \oplus \k_m +i$ as follows:
\begin{itemize}
    \item If $\o_{\k_0 \oplus \k_m +i} \neq \pi^{**}$, follow the stationary strategy $\s_{1-x}^*$, i.e., split $\o_{\k_0 \oplus \k_m +i}$ to $\chi_{1-x}^* (\o_{\k_0 \oplus \k_m +i})$, where $\chi_{1-x}^* : \dk \to \bigcup_{\x \in \dk}  \mathcal{A}_{\xi}$ was described in Subsection \ref{SubSec Opt Str}. Note that for $i=0$ it is always the case that $\o_{\k_0 \oplus \k_m +i} = \pi^* \neq \pi^{**}$.
    \item If $\o_{\k_0 \oplus \k_m +i} = \pi^{**}$, split $\o_{\k_0 \oplus \k_m +i}$ to the distribution $\mu^*_{i+1} (1-x,\pi)$ described in Subsection \ref{Subsec Optimal Distributions}.
\end{itemize}

Before proceeding with the proof, let us provide some intuition to $\theta^*$. First, $\theta^*$ views the sub-games consisting of the play along stages $\k_0 \oplus \k_{m},...,\k_0 \oplus \k_{m+1}-1$, $m\geq 0$, as independent games. In words, those are the games occurring between any two successive memory erasures by chance. As $\k_{m+1}$ is distributed geometrically with parameter $x$, the play in the mentioned games is of random duration. Therefore, the sender may try to guarantee the random duration value $v_x^{\text{r-d}} (\pi) = v_{1-x}(\pi)/x$ (described in Subsection \ref{Subsec Random Duration}) by mimicking the behavior of $\s^*_{1-x}$ along each such game.

In the absence of the adversary, the sender could follow $\s_{1-x}^*$ which by Claim  \ref{Claim Random Duration} to guarantee an expected sum of stage payoffs equal to $v_x^{\text{r-d}} (\pi)$. Therefore, whenever the adversary does not make a move, i.e., $\o_{\k_0 \oplus \k_m +l} \neq \pi^{**}$, we let $\theta^*$ follow $\s^*_{1-x}$.
Alternatively, when $\o_{\k_0 \oplus \k_m +l} = \pi^{**}$, $\theta^*$ is meant to `restore the receiver's memory' by splitting $\pi^{**}$ to $\mu^*_{i+1} (1-x,\pi)$, being the distribution of the $(i+1)$'st posterior under $\s_{1-x}^*$. Note that as $\mu^*_{i+1} (1-x,\pi)$ may be supported on $k^{i+1}$ beliefs, the infinite cardinality of the signal set $S =\N$ is indeed required. It allows the sender to do the so called `memory restorations' regardless of the identity of stages at which the adversary makes his moves.

The following claim establishes formally that $\theta^*$ achieves the above mentioned goal.

\begin{claim}\label{Claim prop1}
 For each integer $m \geq 0$ and $\t \in \T$ we have
\begin{equation}\label{Eq. theta* random duration}
E^x_{\theta^{*},\t} \left(\sum_{i = \k_0 \oplus \k_m}^{\k_0 \oplus \k_{m+1} - 1} f(\o_i,\chi_i,c_i)  \right) = \frac{v_{1-x}(\pi)}{x}.
\end{equation}
\end{claim}
\begin{proof}[Proof of Claim \ref{Claim prop1}.]
By the definition of $\theta^*$, $f$, and the proof details of Claim \ref{Claim Random Duration} we have
\begin{align*}
    E^x_{\theta^{*},\t} \left(\sum_{i = \k_0 \oplus \k_m}^{\k_0 \oplus \k_{m+1} - 1} f(\o_i,\chi_i,c_i) \, \Bigg|\, \k_0 \oplus \k_m \right) & = \frac{1}{x}  \sum_{n=1}^{\infty} x\cdot (1-x)^{n-1} \int_{\dk} u(\x)\, \text{d}[\mu^*_{n} (1-x,\pi)] (\x)\\
    & = \frac{1}{x} \cdot \gamma_{1-x}(\pi, \s^*_{1-x}) = \frac{v_{1-x}(\pi)}{x}.
\end{align*}
The above is sufficient to deduce \eqref{Eq. theta* random duration}, and thus completes the proof. 
\end{proof}

Our goal now will be to estimate $A_N(x, \theta^{*},\t)$ for large values of $N$. To do so, we consider the event $F_N := \{T_N \geq \lceil (N-1)x - N^{0.75} \rceil \}$ and assume that $N$ big enough so that $(N-1)x - N^{0.75} \geq 0$. To simplify the technical presentation let us denote $a_N := \lceil (N-1)x - N^{0.75} \rceil$. Using the non-negativity of $f$, we obtain that for any $\t \in \T$ it holds that
\begin{align}\label{Eq. D deco1}
& A_N(x, \theta^{*},\t) =  E^x_{\theta^*, \t} \left(\frac{1}{N} \sum_{m=0}^{T_N - 1} \left(\sum_{i = \k_0 \oplus \k_m}^{\k_0 \oplus \k_{m+1} - 1} f(\o_i,\chi_i,c_i)\right)\right) \nonumber \\ 
&\geq P^x_{\theta^*,\t} (F_N) E^x_{\theta^*, \t} \left(\frac{1}{N} \sum_{m=0}^{T_N - 1} \left(\sum_{i = \k_0 \oplus \k_m}^{\k_0 \oplus \k_{m+1} - 1} f(\o_i,\chi_i,c_i)\right)\,\Bigg|\, F_N\right) \nonumber \\
& \geq P^x_{\theta^*,\t} (F_N) E^x_{\theta^*, \t} \left(\frac{1}{N} \sum_{m=0}^{ a_N  - 1} \left(\sum_{i = \k_0 \oplus \k_m}^{\k_0 \oplus \k_{m+1} - 1} f(\o_i,\chi_i,c_i)\right)\,\Bigg|\, F_N\right)\\
& \geq E^x_{\theta^*, \t} \left(\frac{1}{N} \sum_{m=0}^{a_N - 1} \left(\sum_{i = \k_0 \oplus \k_m}^{\k_0 \oplus \k_{m+1} - 1} f(\o_i,\chi_i,c_i)\right)\right) 	\quad \quad \quad 	:= I_N^1 (\t)	\nonumber \\
&\quad \quad \quad - P^x_{\theta^*,\t}\, (F^{\text{c}}_N)\, E^x_{\theta^*, \t} \left(\frac{1}{N} \sum_{m=0}^{a_N - 1} \left(\sum_{i = \k_0 \oplus \k_m}^{\k_0 \oplus \k_{m+1} - 1} f(\o_i,\chi_i,c_i)\right)\,\Bigg|\, F^{\text{c}}_N\right)	\quad \quad \quad 	 :=I_N^2 (\t). \nonumber 
\end{align}
The linearity of the expectation operator and Claim \ref{Claim prop1} imply that, for any $\t \in \T$ it holds
\begin{align}\label{Eq. Ug11}
I_N^1 (\t) & = \frac{1}{N} \sum_{m=0}^{a_N - 1 } E^x_{\theta^*, \t} \left(\sum_{i = \k_0 \oplus \k_m}^{\k_0 \oplus \k_{m+1} - 1} f(\o_i,\chi_i,c_i)\right)  = \frac{a_N}{N} \cdot \frac{v_{1-x}(\pi)}{x} \nonumber \\
& \geq \frac{ (N-1)x - N^{0.75} }{N} \cdot \frac{v_{1-x}(\pi)}{x}\\
& \geq \frac{N-1}{N} \cdot v_{1-x}(\pi) - \frac{ v_{1-x}(\pi)}{x N^{0.25}} \geq v_{1-x}(\pi) - \frac{\Vert u \Vert_{\infty}}{N} - \frac{\Vert u \Vert_{\infty}}{x N^{0.25}} \nonumber\\
& \geq v_{1-x}(\pi) - \frac{\Vert u \Vert_{\infty}}{x} \left( \frac{x}{N} + \frac{1}{N^{0.25}} \right) \geq v_{1-x}(\pi) - \frac{\Vert u \Vert_{\infty}}{x} \cdot \frac{2} {N^{0.25}}\nonumber.
\end{align}

Next, let us analyze $I_N^2 (\t)$. First, we have for every $\t \in \T$
\begin{multline}\label{Eq. Ug0}
E^x_{\theta^*, \t} \left(\frac{1}{N} \sum_{m=0}^{a_N - 1} \left(\sum_{i = \k_0 \oplus \k_m}^{\k_0 \oplus \k_{m+1} - 1} f(\o_i,\chi_i,c_i)\right)\,\Bigg|\, F^{\text{c}}_N\right)\\ \leq \frac{\Vert u \Vert_{\infty}}{N}\, E^x_{\theta^*, \t} \left( \sum_{m=0}^{a_N - 1 } \k_{m+1} \,\Big|\, F^{\text{c}}_N\right). 
\end{multline}
We now claim that
\begin{equation*}
F^{\text{c}}_N = \{\k_1 + \cdots + \k_{a_N} > N-1 \}.
\end{equation*} 
Indeed, the event $F^{\text{c}}_N$ occurs when there were less then $a_N = \lceil (N-1)x - N^{0.75} \rceil$ memory erasures performed by chance in the first $N$ stages. As chance moves may occur from the second stage on ($\o_1 = \pi \neq \pi^*$), this means that along stages $2,...,N$ there were less then $a_N = \lceil (N-1)x - N^{0.75} \rceil$ memory erasures. Moreover,  $\k_1 + \cdots + \k_{a_N} - a_N$ follows a negative binomial distribution with parameters $(a_N,x)$ (being the number of stages with no erasures performed by chance until the $a_N$'th erasure performed by chance). Therefore, as we have that  
\begin{multline}\label{Eq. Ug1}
E^x_{\theta^*, \t} \left( \sum_{m=0}^{a_N - 1} \k_{m+1}\,\Big|\, F^{\text{c}}_N\right) =  a_N  \\ + E^x_{\theta^*, \t} \left( \left(\sum_{m=0}^{a_N - 1} \k_{m+1} \right) - a_N \,\Big|\, \k_1 + \cdots + \k_{a_N} > N -1 \right),
\end{multline}
we may now use the following bound by Geyer (e.g., Eq. (16) in Geyer \cite{Geyer}) for the lower truncated mean of a negative binomial distribution $W$ with parameters $(r,p)$:
\begin{align*}
E(W\,|\, W > n ) = E(W) + \frac{n+1}{p(1+\b_{n+1}(r,p))},	\quad \forall n \in \N,
\end{align*}
where $\b_{n+1}(r,p):= P(W > n+1)/\, P(W=n+1)$. Applying a relaxed version of the above (using $1 + \b_{n+1}(r,p) \geq 1$) to our setting we obtain that for every  $\t \in \T$ it holds that
\begin{multline}\label{Eq. Ug2}
E^x_{\theta^*, \t} \left( \left(\sum_{m=0}^{a_N-1} \k_{m+1} \right) - a_N \,\Big|\, \k_1 + \cdots + \k_{a_N} > N-1\right)  \leq \\ \frac{a_N(1-x)}{x} + \frac{N}{x}.
\end{multline} 
Combining  Eqs.\ (\ref{Eq. Ug0}), \eqref{Eq. Ug1} and (\ref{Eq. Ug2}) we obtain that 
\begin{align}\label{Eq. Ug3}
& E^x_{\theta^*, \t} \left(\frac{1}{N} \sum_{m=0}^{a_N - 1} \left(\sum_{i = \k_0 \oplus \k_m}^{\k_0 \oplus \k_{m+1} - 1} f(\o_i,\chi_i,c_i)\right)\,\Bigg|\, F^{\text{c}}_N\right) \nonumber\\
& \quad \leq \frac{\Vert u \Vert_{\infty}}{N} \cdot \left( a_N + \frac{a_N(1-x)}{x} + \frac{N}{x}\right) \nonumber \\
& \quad  \leq \frac{\Vert u \Vert_{\infty}}{N} \cdot \left(Nx+1 + \frac{(Nx+1)(1-x)}{x} + \frac{N}{x} \right)\\
& \quad \leq \frac{\Vert u \Vert_{\infty}}{N} \cdot \left( 2N + \frac{2N}{x} + \frac{N}{x} \right)  = \Vert u \Vert_{\infty}\left(2 + \frac{3}{x}\right),\nonumber
\end{align}
where the second inequality follows from the fact that $a_N = \lceil (N-1)x - N^{0.75} \rceil \leq Nx+1$ and last inequality uses the relation $Nx+1 \leq 2N$. The last step in estimating $I_N^2(\t)$ consists of estimating $P^x_{\theta^*,\t}\, (F^{\text{c}}_N)$. Using Hoeffding's inequality for the sum of i.i.d.\ Bernoulli random variables $T_N$ (see relation \eqref{T_N Bernoulli Representation}) we obtain that for every $\t \in \T$ it holds that
\begin{align}\label{Hoef 1}
P^x_{\theta^*,\t}\, (F^{\text{c}}_N) & = P^x_{\theta^*,\t}\, \left( \{T_N < \lceil (N-1)x - N^{0.75} \rceil\}\right) \nonumber\\
&= P^x_{\theta^*,\t}\, \left(\{ T_N <  (N-1)x - N^{0.75}\}\right) \nonumber \\
& \leq P^x_{\theta^*,\t}\left( \{\vert T_N-(N-1)x \vert \geq N^{0.75}\} \right)\\
& \leq  2 \exp \left( \frac{-2 N^{1.5}}{N-1} \right)\nonumber\\
& \leq 2 \exp \left( -2\sqrt{N} \right) \nonumber
\end{align}
where the second equality follows from the fact that $T_N$ is integer-valued. 
Thus, combining Eqs. (\ref{Eq. Ug3}) and  (\ref{Hoef 1}) we have that
\begin{align*}
    I_N^2 (\t) & := P^x_{\theta^*,\t}\, (F^{\text{c}}_N)\, E^x_{\theta^*, \t} \left(\frac{1}{N} \sum_{m=0}^{a_N - 1} \left(\sum_{i = \k_0 \oplus \k_m}^{\k_0 \oplus \k_{m+1} - 1} f(\o_i,\chi_i,c_i)\right)\,\Bigg|\, F^{\text{c}}_N\right)\\
    & \leq 2 \exp \left( -2\sqrt{N} \right) \cdot \Vert u \Vert_{\infty}\left(2 + \frac{3}{x}\right).
\end{align*}
The latter combined with Eqs.\ \eqref{Eq. D deco1} and \eqref{Eq. Ug11} implies
\begin{align*}
A_N(x, \theta^{*},\t) & \geq v_{1-x}(\pi) - \frac{\Vert u \Vert_{\infty}}{x} \cdot \frac{2}{N^{0.25}} - 2 \exp \left( -2\sqrt{N} \right) \cdot \Vert u \Vert_{\infty}\left(2 + \frac{3}{x}\right)\\
& = v_{1-x}(\pi) - \frac{\Vert u \Vert_{\infty}}{x}\, \left( \frac{2}{N^{0.25}} + 2 \exp \left( -2\sqrt{N} \right) \left(2x+3\right)\right),
\end{align*}
which in conjunction with Eq.\ (\ref{Eq. approx.}) yields that 
\begin{equation*}
\gamma_N^{x} (\theta^{*},\t) \geq v_{1-x}(\pi) -  \frac{\Vert u \Vert_{\infty}}{x}\, \left( \frac{2}{N^{0.25}} + 2 \exp \left( -2\sqrt{N} \right) \left(2x+3\right) + \frac{1-x}{N} \right)
\end{equation*}
for every $\t \in \T$. The validity of Proposition \ref{prop1} may now be derived as \begin{equation*}
\frac{\Vert u \Vert_{\infty}}{x}\, \left( \frac{2}{N^{0.25}} + 2 \exp \left( -2\sqrt{N} \right) \left(2x+3\right) + \frac{1-x}{N} \right) \to 0
\end{equation*}
as $N \to \infty$.

We move on to the proof of Proposition \ref{prop2}.

\subsection{Proof of Proposition \ref{prop2}}\label{Subsec Prop 2}

Fix $y \in (x,1]$ and denote $\b_y := (y-x)/(1-x)$. Consider the following strategy $\t^{y} \in \T$: at any stage $n\geq 1$, 
\begin{align*}
    \t^y_n : (c_1,...,c_{n-1}) \mapsto (1-\b_y) \cdot \d_{\{0\}} + \b_y \cdot \d_{\{1\}}, \quad \forall (c_1,...,c_{n-1}) \in \{0,1\}^{n-1}.
\end{align*}
In words, in every stage $n\geq 1$, the adversary performs a memory erasure with probability $(y-x)/(1-x)$, independently of his past moves. 

For each $n\geq 2$ let $Z_n := \textbf{1}\{ \o_n \in \{\pi^*, \pi^{**}\} \}$. By the definition of $\t_y$, for any $\theta \in \Theta$, we have that $(Z_n)_{n\geq 2}$ forms an i.i.d.\ Bernoulli process with parameter $y$ on the the probability space $(H, \mathcal{H}, P^x_{\theta,\t^y})$. Indeed, the identical distribution follow from the fact that 
\begin{align*}
     P^x_{\theta,\t^y}\left(\{Z_n =1\} \right) & = P^x_{\theta,\t^y}\left(\{\o_n = \pi^*\} \right) + P^x_{\theta,\t^y}\left(\{\o_n \neq  \pi^*\} \cap \{\o_n =  \pi^*\} \right)\\
     & = x + (1-x)P^x_{\theta,\t^y} \left( \{\o_n =  \pi^*\} \,|\, \{\o_n \neq  \pi^*\} \right)\\
     & x + (1-x) \cdot \frac{y-x}{1-x} = x + y-x =y.
\end{align*}
The fact that the random variables $(Z_n)_{n\geq 2}$ are independent on $(H, \mathcal{H}, P^x_{\theta,\t^y})$ follows from the definition of $\t^y$, stating that the moves of the adversary are independent of past events. 

Let us now generalize the probabilistic framework considered in Subsection \ref{Subsec Probabilistic Framework} for a given $\theta \in \Theta$. First, define by iteration the sequence of random variables $(\k^*_n)_{n \geq 0}$ on $(H, \mathcal{H}, P^x_{\theta,\t^y})$ according to:
\begin{align*}
\k^*_0 & := 1, \nonumber \\
\k^*_n & := \inf\,\{\, n \geq 1\,:\, Z_{\k^*_0 + \k^*_1 \cdots + \k^*_{n-1}+n} = 1\, \} \quad  \\\
       &\,\,\, =  \inf\,\{\, n \geq 1\,:\, \o_{\k^*_0 + \k^*_1 \cdots + \k^*_{n-1}+n}  \in \{\pi^*, \pi^{**}\}\, \}, \quad  n \geq 1.
\end{align*}
The sequence $(\k^*_n)_{n\geq 1}$ counts the number of stages between successive memory erasures performed by either chance or the adversary. Since $(Z_n)_{n\geq 2}$ are i.i.d.\ Bernoulli random variables with parameter $y$ on $(H, \mathcal{H}, P^x_{\theta,\t^y})$, we have that under $P^x_{\theta,\t^y}$, $(\k^{*}_n)_{n \geq 1}$ are i.i.d.\ geometric random variables with parameter $y$.

Next, define for any $N \in \N$ the random integer $T^{*}_N$ by
\begin{align*}
T^*_N := \max\, \{\, n \in \{0,1,..., N-1\} \, : \, \k^*_0 + \k^*_1+ \cdots + \k^*_n \leq N\, \}.
\end{align*}
The random variable $T_N^*$, which also admits the representation $
T^{*}_N := Z_2+\cdots+Z_N$, counts the total number of memory erasures by either chance or the adversary in the first $N$ stages of $\G$. Lastly, let $\l^{*}_N := \k^{*}_0+\k^{*}_1+\cdots + \k^{*}_{T^{*}_N}$ be the last stage within the first $N$ stages at which a memory erasure was performed by either chance or the adversary.

In terms of the new probabilistic framework, the payoff $\g_N^{x} (\theta,\t^{y})$ can now be written as follows:
\begin{align*}
\gamma_N^x (\theta,\t^y) & = E^x_{\theta,\t^y} \left(\frac{1}{N}\sum_{n=1}^{\l^*_N} f(\o_n,\chi_n,c_n) \right) \quad \quad \quad := A^*_N (x,\theta,\t^y)  \\
& \quad \quad \quad + E^x_{\theta,\t^y} \left(\frac{1}{N}\sum_{n= \l^*_N +1}^{N} f(\o_n,\chi_n,c_n) \right) \quad \quad \quad := B^*_N (x,\theta,\t^y). \nonumber
\end{align*}
Using an adapted version of the arguments in the probabilistic analysis that leading to Eq.\ (\ref{Eq. approx.}), we have that in the current setup it holds
\begin{align}\label{Eq. approx.2}
\vert \g_N^{x} (\theta,\t^{y}) - A^{*}_N (x,\theta,\t^{y}) \vert \leq \frac{1-y}{y} \cdot \frac{\Vert u \Vert_{\infty}}{N}.
\end{align}
Our next goal now is to estimate $A^{*}_N (x,\theta,\t^{y})$ for large values of $N$. To do so, we consider the events $E_N := \{T^{*}_N \leq b_N\}$ where $b_N := \lceil (N-1)y + N^{0.75} \rceil\}$. Since $\l^{*}_N \leq N$ for every $N$, and $f$ is non-negative we have
\begin{align}\label{Eq. Decompos D2}
A^{*}_N (x,\theta,\t^{y}) & =  E^x_{\theta, \t^y} \left(\frac{1}{N} \sum_{m=0}^{T^*_N - 1} \left(\sum_{i = \k^*_0 \oplus \k^*_m}^{\k^*_0 \oplus \k^*_{m+1} - 1} f(\o_i,\chi_i,c_i)\right)\right) \nonumber \\ 
& \leq P^x_{\theta,\t^{y}} (E_N)\, E^x_{\theta,\t^{y}} \left(\frac{1}{N} \sum_{m=0}^{T^*_N - 1} \left(\sum_{i = \k^*_0 \oplus \k^*_m}^{\k^*_0 \oplus \k^*_{m+1} - 1} f(\o_i,\chi_i,c_i)\right)\,\Bigg|\, E_N\right) \nonumber\\
& 	\quad \quad \quad  + P^x_{\theta,\t^{y}} (E_N^{\text{c}}) \Vert u \Vert_{\infty} \nonumber \\
& \leq P^x_{\theta,\t^{y}} (E_N)\, E^x_{\theta,\t^{y}} \left(\frac{1}{N} \sum_{m=0}^{b_N - 1} \left(\sum_{i = \k^*_0 \oplus \k^*_m}^{\k^*_0 \oplus \k^*_{m+1} - 1} f(\o_i,\chi_i,c_i)\right)\,\Bigg|\, E_N\right) \nonumber \\
& 	\quad \quad \quad  + P^x_{\theta,\t^{y}} (E_N^{\text{c}}) \Vert u \Vert_{\infty}  \\
& \leq E^x_{\theta,\t^{y}} \left(\frac{1}{N} \sum_{m=0}^{b_N - 1} \left(\sum_{i = \k^*_0 \oplus \k^*_m}^{\k^*_0 \oplus \k^*_{m+1} - 1} f(\o_i,\chi_i,c_i)\right)\right) 	 + P^x_{\theta,\t^{y}} (E_N^{\text{c}}) \Vert u \Vert_{\infty}. \nonumber 
\end{align}
We proceed by estimating $P^x_{\theta,\t^{y}} (E_N^{\text{c}})$. Using Hoeffding's inequality we obtain that for every $N \geq 1$  it holds that
\begin{align}\label{Hoeff 2}
P^x_{\theta,\t^{y}} (E_N^{\text{c}}) & = P^x_{\theta,\t^{y}}\left( \{T_N^* >\lceil (N-1)y + N^{0.75} \rceil \}\right) \nonumber\\
& \leq  P^x_{\theta,\t^{y}}\left( \{T_N^* > (N-1)y + N^{0.75} \}\right) \nonumber\\
& \leq P^x_{\theta,\t^{y}} \left( \{\vert T^{*}_N - (N-1)y \vert > N^{0.75}\} \right)\\
& \leq  2 \exp \left( \frac{-2 N^{1.5}}{N-1} \right) \nonumber \\
& \leq  2 \exp \left( -2\sqrt{N} \right). \nonumber
\end{align}
Combining Eqs.\ \eqref{Eq. Decompos D2} and \eqref{Hoeff 2}, and using the linearity of $E^x_{\theta,\t^y}$ we obtain that 
\begin{align}\label{Eq. Prop2 Bound 1}
    A^{*}_N (x,\theta,\t^{y}) \leq \frac{1}{N} \sum_{m=0}^{b_N - 1} E^x_{\theta,\t^y} \left( \sum_{i = \k^*_0 \oplus \k^*_m}^{\k^*_0 \oplus \k^*_{m+1} - 1} f(\o_i,\chi_i,c_i) \right) +  2 \exp \left( -2\sqrt{N} \right) \cdot \Vert u \Vert_{\infty}.
\end{align}

We proceed with the following intuitive lemma. 
\begin{claim}\label{Claim Upper Bound within R-D block}
    For every $m\geq 0$ it holds that 
    \begin{align}\label{Eq. Claim Upper Bound within R-D block}
           E^x_{\theta,\t^y} \left( \sum_{i = \k^*_0 \oplus \k^*_m}^{\k^*_0 \oplus \k^*_{m+1} - 1} f(\o_i,\chi_i,c_i)\, \Bigg| \, \k^*_0 \oplus \k^*_m \right)  \leq \frac{v_{1-y}(\pi)}{y}.
    \end{align}
\end{claim}

\begin{proof}[Proof of Claim \ref{Claim Upper Bound within R-D block}]
    Consider the event $\{\k^*_0 \oplus \k^*_m=l\}$ for some positive integer $l$. Conditional on this event, the play along stages $l,...,l+\k^*_{m+1} - 1$ follows the rules of the MDP described in Subsection \ref{Subsec MDP reformulation} starting from the prior $\pi$ with random duration of geometric parameter $(1-y)$. Nevertheless, the sender may base his actions along stages $l,...,l+\k^*_{m+1} - 1$ on events that occurred along the first $l-1$ stages of $\G$. 
    
    This fact does not help the sender to guarantee an expected payoff higher then the random duration value under the MDP setup, which in view of Claim \ref{Claim Random Duration} equals to $v_{1-y}(\pi)/y$. Indeed, along stages $l,...,l+\k^*_{m+1} - 1$ each strategy $\theta$ can be viewed as a mixed strategy in the MDP (see Subsection \ref{Subsec MDP reformulation}). It chooses a pure strategy in the MDP, based on the history $\o_1,\chi_1,...,\o_{l-1},\chi_{l-1}$ that occurred in the first $l-1$ stages (where $\o_l = \pi^{**}$). Such history is drawn at random by $P^x_{\theta, \t^y}$. However, as already mentioned in regards to the MDP in Subsection \ref{Subsec MDP reformulation}, the $(1-y)$-discounted value is attained at a pure stationary strategy, and thus the sender cannot benefit from using mixed strategies. 
\end{proof}

We proceed with the proof of Proposition \ref{prop2}. Implementing the result of Claim \ref{Claim Upper Bound within R-D block} to Eq.\ \eqref{Eq. Prop2 Bound 1} we obtain
\begin{align*}
    A^{*}_N (x,\theta,\t^{y}) & \leq \frac{1}{N} \sum_{m=0}^{b_N - 1} \frac{v_{1-y}(\pi)}{y} +  2 \exp \left( -2\sqrt{N} \right) \cdot \Vert u \Vert_{\infty}\\
    & = \frac{b_N}{N} \cdot \frac{v_{1-y}(\pi)}{y} +  2 \exp \left( -2\sqrt{N} \right) \cdot \Vert u \Vert_{\infty}\\
    & \leq v_{1-y}(\pi) - \left( 1 - \frac{b_N}{Ny}\right)\cdot\Vert u \Vert_{\infty} + 2 \exp \left( -2\sqrt{N} \right) \cdot \Vert u \Vert_{\infty} \\
    & \leq  v_{1-y}(\pi) - \left(\frac{Ny - (N-1)y - N^{0.75}}{Ny}\right)\cdot\Vert u \Vert_{\infty} + 2 \exp \left( -2\sqrt{N} \right) \cdot \Vert u \Vert_{\infty}\\
    &\leq  v_{1-y}(\pi) - \left(\frac{y - N^{0.75}}{Ny}\right)\cdot\Vert u \Vert_{\infty} + 2 \exp \left( -2\sqrt{N} \right) \cdot \Vert u \Vert_{\infty} \\
    & \leq v_{1-y}(\pi) + \left( \frac{1}{yN^{0.25}} + 2 \exp \left( -2\sqrt{N}  \right) \right)\cdot \Vert u \Vert_{\infty}.
\end{align*}
The above together with Eq.\ \eqref{Eq. approx.2} implies that
\begin{align*}
    \g_N^{x} (\theta,\t^{y}) \leq v_{1-y}(\pi) + \left( \frac{1}{yN^{0.25}} + 2 \exp \left( -2\sqrt{N}  \right) \right)\cdot \Vert u \Vert_{\infty} + \frac{1-y}{y} \cdot \frac{\Vert u \Vert_{\infty}}{N}.
\end{align*}
This completes the proof of Proposition \ref{prop2} as 
\begin{align*}
    \left( \frac{1}{yN^{0.25}} + 2 \exp \left( -2\sqrt{N}  \right) \right)\cdot \Vert u \Vert_{\infty} + \frac{1-y}{y} \cdot \frac{\Vert u \Vert_{\infty}}{N} \to 0 \quad \text{as} \quad N \to \infty.
\end{align*}

\section{Proof of Theorem \ref{Thm3}}\label{Sec Proof of Thm3}

The cornerstone property required for the proof is stated in the next Lemma.

\begin{lemma}\label{Lemma Thm3 - 1}
    For every $\d \in (0,1)$ and $\x \in \dk$ it holds
    \begin{align*}
        \vert v_{\d} (\x) - v_{\d} (\x M)\vert \leq 
        \frac{(1-\d) \Vert u \Vert_{\infty}}{\d}.
    \end{align*}
\end{lemma}

\begin{proof}[Proof of Lemma \ref{Lemma Thm3 - 1}]
    Fix $\x$ and $\d$. Recall that by the signal cardinality invariance principal we may take the signal set $S = \N$ (see Subsection \ref{SICP property}) Consider first the case where the prior $q_1$ in the Markovian persuasion model is $\x$. Then, as the sender can keep mute at the first stage and plays optimally from the second stage on we must have that
    \begin{align}\label{Eq. 3.1}
        v_{\d} (\x) & \geq (1-\d) u(\x) + \d v_{\d} (\x M) \nonumber \\
        & = v_{\d} (\x M) - (1-\d) v_{\d} (\x M) + (1-\d)u(\x)\\
        & \geq v_{\d} (\x M) - (1-\d) v_{\d} (\x M) \geq v_{\d} (\x M) - \Vert u \Vert_{\infty} (1-\d), \nonumber
    \end{align}
    where the second inequality follows from the non-negativity of $u$.

    On the opposite side consider the case where the prior $q_1 = \x M$. Let us define the signaling policy $\s^* \in \Sigma$ as follows:
    \begin{itemize}
        \item At the first stage split $\x M$ to $\mu^*_2 (\d,\x)$, where $\mu^*_2 (\d,\x)$ was defined in Subsection \ref{Subsec Optimal Distributions}.
        \item Starting from the second stage on follow the optimal stationary strategy $\chi_{\d}^* :\dk \to \bigcup_{\x \in \dk} \mathcal{A}_{\xi}$ defined in Subsection \ref{Sub-sub sec Functional Equation}.
    \end{itemize}
    By the definition of $\s^*$ we have
    \begin{align*}
        \g_{\d} (\x M, \s^*) & = E_{\x, \chi_{\d}^*} \left( (1-\d) \sum_{n=1}^{\infty} \d^{n-1} u(p_{n+1}) \right)\\
        & = E_{\x, \chi_{\d}^*} \left( (1-\d) \sum_{n=2}^{\infty} \d^{n-2} u(p_{n}) \right)\\
        & = \frac{1}{\d} \left( v_{\d} (\x) - (1-\d) E_{\x, \chi_{\d}^*} u(p_{1}) \right)\\
        & \geq \frac{1}{\d} \left( v_{\d} (\x) - (1-\d) \Vert u \Vert_{\infty} \right) \geq v_{\d} (\x) - \frac{(1-\d)}{\d} \Vert u \Vert_{\infty},
        \end{align*}
        where in the last inequality we used the non-negativity of $v_{\d}$. Therefore we obtain that $v_{\d}(\x M) \geq  \g_{\d} (\x M, \s^*) \geq v_{\d} (\x) - \frac{(1-\d)}{\d} \Vert u \Vert_{\infty}$, which together with Eq.\ \eqref{Eq. 3.1} suffices for the proof of the Lemma.
    \end{proof}

    \begin{corollary}\label{Corollary 3.1}
        For every $n\geq 1$, $\x \in \dk$ and $\d \in (0,1)$ it holds
        \begin{align*}
        \vert v_{\d} (\x) - v_{\d} (\x M^n)\vert \leq 
        \frac{n (1-\d) \Vert u \Vert_{\infty}}{\d}.
    \end{align*}
    \end{corollary}

    \begin{proof}[Proof of Corollary \ref{Corollary 3.1}]
        Using the triangle inequality and Lemma \ref{Lemma Thm3 - 1} we obtain
        \begin{align*}
        \vert v_{\d} (\x) - v_{\d} (\x M^n)\vert \leq \sum_{i=1}^{n} \vert v_{\d} (\x M^{i-1}) - v_{\d} (\x M^i)\vert \leq \frac{n (1-\d) \Vert u \Vert_{\infty}}{\d}
        \end{align*}
        as desired. 
    \end{proof}

    Let $\Vert \cdot \Vert_1$ denote the $\l_1$-norm on $\dk$, i.e., $\Vert \x - p \Vert_1 = \sum_{\l \in K} \vert \x^{\l} - p^{\l} \vert$ for $\x, p \in \dk$. For every $\x \in \dk$ and $\a>0$ set $B_{\l_1} (\x,\a):= \{p \in \dk: \Vert p-\x\Vert_1 \leq \a \}$ to be the $\l_1$-norm ball of radius $\a$ around $\x$.

    The following Lemma concerns a continuity property of $v_{\d}$ around $\p$. For its statement let us introduce the notation $c_* := \min_{\l \in K} \p^{\l}$, and note that by the ergodicity assumption we have $c_*>0$.

    \begin{lemma}\label{Lemma Thm3 - 2}
        For every two beliefs $\x,p \in B_{\l_1}(\p, c_{*}/2)$ and $\d \in [0,1)$ it holds
        \begin{align}\label{Eq. local Lipschitz}
        \vert v_{\d}(\x) - v_{\d}(p) \vert < \frac{2 \Vert u \Vert_{\infty}}{c_*}  \Vert \x - p \Vert_1.
        \end{align}
    \end{lemma}

  \begin{proof}[Proof of Lemma \ref{Lemma Thm3 - 2}]
  Denote by $\partial C$ the boundary of a closed subset $C \subseteq \dk$. Since $\Vert q - \p\Vert_1 > c_* $ for every $q \in \partial\, \dk$ it holds that 
  \begin{align}\label{Eq. Inclusion of l_1 ball}
      \Vert \p - q \Vert_1 = c_*, \quad \forall q \in \partial B_{\l_1} (\p, c_*).
  \end{align}
  Take $p,\x \in B_{\l_1}(\p, c_{*}/2)$. Let $z \in \partial B_{\l_1} (\p, c_*)$ be the belief in the intersection of the ray $\{ \x + \a (p-\x): \a \in \R_{+}\}$ with $\partial B_{\l_1} (\p, c_*)$. It holds that 
  \begin{align*}
      p = \frac{\Vert z- p\Vert_1}{\Vert \x- z\Vert_1} \cdot \x + \frac{\Vert \x- p\Vert_1}{\Vert \x- z\Vert_1} \cdot z.
  \end{align*}
  As $v_{\d}$ is concave on $\dk$ for every $\d \in [0,1)$ we obtain that
  \begin{align*}
      v_{\d} (p) & \geq \frac{\Vert z- p\Vert_1}{\Vert \x- z\Vert_1} \cdot v_{\d} (\x) + \frac{\Vert \x- p\Vert_1}{\Vert \x- z\Vert_1}\cdot v_{\d}(z)\\
      & = v_{\d} (\x) + \frac{\Vert \x- p\Vert_1}{\Vert \x- z\Vert_1} (v_{\d}(z) - v_{\d} (\x))\\
      & \geq v_{\d} (\x) - \Vert u \Vert_{\infty} \frac{\Vert \x- p\Vert_1}{c_* /2},
  \end{align*}
  where the last inequality follows from $0 \leq v_\d \leq \Vert u \Vert_{\infty}$ and the inequality $\Vert \x- z\Vert_1 \geq c_* /2$ which holds due to relation \eqref{Eq. Inclusion of l_1 ball}. Therefore we have obtained that 
  \begin{align*}
     v_{\d} (p) - v_{\d} (\x) \geq -  \frac{2 \Vert u \Vert_{\infty}}{c_*} \Vert \x- p\Vert_1, \quad \forall \x, p \in B_{\l_1}(\p, c_{*}/2).
  \end{align*}
  Changing the roles of $\x$ and $p$ we obtain that Eq.\ \eqref{Eq. local Lipschitz} holds, thus proving the Lemma.
  \end{proof}

The last preliminary result required for the proof of Theorem \ref{Thm3} is stated in the following Lemma.

\begin{lemma}\label{Lemma Mixing bound}
    For $\d \in (1 - c_{*}/2,1)$ and every $\x \in \dk$ it holds that:
    \begin{align*}
       \sup_{\x \in \dk} \vert v_{\d}(\x) - v_{\d} (\p)\vert \leq  O\left((1-\d)\log_2 \left(\frac{1}{1-\d}\right)\right).
    \end{align*}
\end{lemma}

\begin{proof}[Proof of Lemma \ref{Lemma Mixing bound}]
    By an upper bound on the mixing time of Markov chains due to Levin and Peres (see Section 4.5 and 4.7 in Levin and Peres \cite{Peres}) there exists a constant $C$ that depends only on $M$ and an integer $N_{\d} \leq C \lceil \log_2 \frac{1}{1-\d}\rceil$ such that $\x M^{N_{\d}} \in B_{\l_1} (\p, 1-\d)$ for every $\d \in [0,1)$. As we assumed that $\d > 1- c_*/2$, we have that $1-\d < c_*/2$, and thus using Lemma \ref{Lemma Thm3 - 2} and Corollary \ref{Corollary 3.1} we obtain that
    \begin{align*}
        \sup_{\x \in \dk}\vert v_{\d}(\x) - v_{\d} (\p)\vert & \leq \sup_{\x \in \dk} \vert v_{\d}(\x) - v_{\d} (\x\, M^{N_{\d}})\vert + \sup_{\x \in \dk} \vert v_{\d}(\xi\, M^{N_{\d}}) - v_{\d} (\p)\vert \\
        & \leq \frac{N_{\d} (1-\d) \Vert u \Vert_{\infty}}{\d} + \frac{2 \Vert u \Vert_{\infty}}{c_*}(1-\d)\\
        & \leq \frac{C\Vert u \Vert_{\infty}}{1-c_{*}/2} \lceil \log_2 \frac{1}{1-\d}\rceil (1-\d) + \frac{2 \Vert u \Vert_{\infty}}{c_*}(1-\d).
    \end{align*}
    As the function $$\d \to \frac{C\Vert u \Vert_{\infty}}{1-c_{*}/2} \lceil \log_2 \frac{1}{1-\d}\rceil (1-\d) + \frac{2 \Vert u \Vert_{\infty}}{c_*}(1-\d) $$
    is of the order of magnitude of the function $\d \mapsto (1-\d)\log_2 \left(\frac{1}{1-\d}\right)$ on $[1-c_*/2,1)$ we complete the proof of the Lemma.
\end{proof}

We are now in position to deduce Theorem \ref{Thm3}. By Eq.\ \eqref{Eq. Uniform Convergence}, Theorems \ref{Thm1} and \ref{Thm2}, and Lemma \ref{Lemma Mixing bound} we have for any $\d > 1 - c_{*}/2$
\begin{align*}
    0\leq v_{\d} (\p) - v_{\infty} & \leq  v_{\d} (\p) - \Psi(\d)\\
                                   & = \sum_{\l \in K} \p^{\l} \cdot (v_{\d} (\p) - v_{\d}(\textbf{e}_{\l}))\\
                                   & \leq O\left((1-\d)\log_2 \left(\frac{1}{1-\d}\right)\right).
\end{align*}
Using the latter and Lemma \ref{Lemma Mixing bound} once more we obtain 
\begin{align*}
    \sup_{\x \in \dk} \vert v_{\d} (\x) -  v_{\infty} \vert & \leq \sup_{\x \in \dk} \vert v_{\d} (\x) -  v_{\d}(\p) \vert + \vert v_{\d} (\p) -  v_{\infty} \vert \\
    & \leq O\left((1-\d)\log_2 \left(\frac{1}{1-\d}\right)\right) + O\left((1-\d)\log_2 \left(\frac{1}{1-\d}\right)\right) \\
    & = O\left((1-\d)\log_2 \left(\frac{1}{1-\d}\right)\right)
\end{align*}
for $\d \in (1-c_{*}/2,1)$, thus completing the proof of Theorem \ref{Thm3}.

\section{Proof of Theorem \ref{Thm RG-1,3}}\label{Sec RG proof}

\subsection{Optimal Stationary Strategies via MDP Reduction}\label{MDP-2}
Similarly to the Markovian persuasion model, our starting point is to reduce the repeated game model to an MDP. Such a reduction builds on known arguments and observations showcased in Renault \cite{Renault}.

The MDP is over the state space $\dk$. The set of actions $\A_{\x}$ at each state $\x \in \dk$ is $\Delta(I)^K$. Namely, the decision maker in the MDP (Player 1 in this case) may assign to each state in $K$ a mixed action over $I$. For each action $\chi \in \Delta(I)^K$, let $\chi^{\l} \in \Delta(I)$ denote the $\l$'th coordinate of $\chi$, where $\l \in K$. Next, for every $\x \in \dk$, $\chi \in \Delta(I)^K$, denote by $\chi(\x) \in \Delta(I)$ the probability measure defined by $\chi(\x)[i] := \sum_{\l \in K} \x^{\l} \cdot \chi^{\l}[i]$ for every $i\in I$. In terms of the original model, $\chi(\x)[i]$ describes the total probability of choosing action $i \in I$, given the belief $\x$ and one-stage strategy $\chi$. In the case where $\chi(\x)[i]>0$ define $\x(\chi,i) \in \dk$ by 
\begin{align*}
    \x(\chi,i)[\l] := \frac{\x^{\l}\cdot \chi^{\l}[i]}{\chi(\x)[i]}, \quad \forall \l \in K.
\end{align*}
If $\chi(\x)[i]=0$ define $\x(\chi,i) \in \dk$ arbitrarily. In terms of the original model, the belief $\x(\chi,i) \in \dk$ corresponds to the posterior belief of Player 2 over $K$ upon observing the action $i$, given that his prior belief was $\x \in \dk$, and given that he knew the one-shot strategy $\chi$ of Player 1. Thus, just in the Markovian persuasion model, given that the posterior belief equals $\x(\chi,i)$, the prior belief in the next stage is shifted to $\x(\chi,i) M$. For this reason we set the stochastic transition rule $\rho$ in the current MDP by
    \begin{align*}
        \rho (\xi,\chi) = \x(\chi,i) M~~ \text{with probability}~~ \chi(\x)[i].
    \end{align*}
The reward function $r$ assigns to each state $\xi$ and action $\chi$, the number $r(\xi,\chi)$, defined by
    \begin{align}\label{Payoff G_x}
        r(\xi,\chi) := \min_{j \in J} \sum_{\l \in K} \x^{\l} \sum_{i \in I} \chi^{\l}[i] G^{\l}(i,j).
    \end{align}
The number $r(\xi,\chi)$ describes the expected stage payoff of Player 1 when Player 2 chooses his optimal response to the one-stage strategy $\chi$ of Player 1 at the belief $\x$. Note that by relation \eqref{RG minimax relation} we may assume that Player 2 knows the strategy $\a \in \A$ of Player 1, and therefore he can play his optimal responses at every stage, based on his prior belief at that stage. 

To summarize in different terms, the MDP reformulates the repeated game as an optimization problem of Player 1 only, where Player 2 is now substituted with a non-strategic device that counters each action of Player 1 with an optimal response at the given state as reflected in \eqref{Payoff G_x}. The pure strategies and mixed strategies in such MDP are defined in the same fashion as in Subsection \ref{Subsec MDP reformulation}.

\bigskip

The MDP reformulation allows us to apply the dynamic programming principle, just as in the Markovian persuasion model. Using such principle (e.g., Theorem 1.22 in \cite{SolanB}) we obtain for every $\x \in \dk$ the following recursive formula for $V_{\d}(\x)$:
\begin{align}\label{Belmann3}
V_{\delta}(\xi)  = \sup_{\chi \in \Delta(I)^K} \bigg\{ (1-\delta)\cdot r(\xi,\chi) + \delta \cdot  \sum_{i \in I} \chi(\x)[i] \cdot V_{\delta}(\x(\chi,i) M) \bigg\}.
\end{align}
The function $V_{\delta}:\dk \to \R$ is known to be concave, and thus continuous. Therefore, by definition, for every $\x \in \dk$ the mapping $\chi \mapsto \chi(\x)[i] \cdot V_{\delta}(\x(\chi,i) M)$ is also continuous on the domain $\Delta(I)^K$. This together with the continuity of the mapping $\chi \mapsto r(\xi,\chi)$ on $\Delta(I)^K$, implies the supremum in Eq.\ \eqref{Belmann3} is attained for every $\x \in \dk$. Hence, for every $\d \in [0,1)$ there exists an optimal stationary strategy of Player 1, which we'll denote by $\hat{\chi}_{\d}$, where $\hat{\chi}_{\d}: \dk \to \Delta(I)^K$ is selected so that
\begin{align*}
    \hat{\chi}_{\d}(\x) \in \underset{\chi \in \Delta(I)^K}{\argmax} \bigg\{ (1-\delta)\cdot r(\xi,\chi) + \delta \cdot  \sum_{i \in I} \chi(\x)[i] \cdot V_{\delta}(\x(\chi,i) M) \bigg\}, \quad \forall \x \in \dk. 
\end{align*}
\bigskip 

To conclude the current discussion, let us define for every $q \in \dk$ and $n\geq 1$, the random variable $\hat{q}_n (q,\d) \in \dk$ by
\begin{align*}
    (\hat{q}_n (q,\d))[\l] = P^q_{\hat{\chi}_{\d},\b} \left( X_n = \l\,|\, i_1,j_1,...,i_{n-1},j_{n-1} \right), \quad \forall \l \in K.
\end{align*}
The random variable $\hat{q}_n (q,\d)$ describes the prior belief of Player 2 over $K$ at the start of stage $n$, given that Player 1 follows the strategy $\hat{\chi}_{\d}$. A well-known fact in the literature on repeated games with incomplete information on one-side is that the definition of  $\hat{q}_n (q,\d)$ is independent of the choice of $\b \in \mathcal{B}$. Moreover, by the definition of $\hat{\chi}_{\d}$, the value of $\hat{q}_n (q,\d)$ is also independent of the previous actions $j_1,...,j_{n-1}$ of Player 2. For every $n\geq 1$, denote by $\hat{\mu}^*_n (\d,q) \in \Delta(\dk)$ the distribution of $\hat{q}_n (q,\d)$:
\begin{align}\label{Optimal Priors Dist}
  \hat{\mu}^*_n (\d,q) (E) := P^q_{\hat{\chi}_{\d},\b}\, (\{\hat{q}_n (q,\d) \in E \}), \quad  \forall ~\text{Borel} ~ E \subseteq \dk.
\end{align}
An important property in regards to such distribution is that $\hat{\mu}^*_n (\d,\x)$ has mean $\x M^{n-1}$, for every $n$ and $\d$. Put differently, $\hat{\mu}^*_n (\d,\x)$ is a split of $\x M^{n-1}$. In particular, for every $n$,  $\hat{\mu}^*_n (\d,\pi)$ is a split of $\pi$, whenever $\pi$ is an invariant distribution of $M$.

\subsection{The Modification of $\G$}\label{Subsec:Gx.Modification}

In order to prove Theorem \ref{Thm RG-1,3} we will generalize the above MDP to a modified version of the game $\G$ (see Subsection \ref{Subsec Stochastic Game with Signals}) in which Player 1 will essentially take the role of the sender in the original version. The modified version of $\G$ will also have $\Omega = \dk \cup \{\pi^*, \pi^{**}\}$ as its state space. However, in the current setup we define the action set $\A_{\x} = \Delta(I)^K$ for every $\x \in \O$. The adversary's binary action set $C = \{0,1\}$ is kept as before.

As it pertains to the payoff function, we define $f: \O \times \Delta(I)^K \times C \to \R$ by \begin{align}\label{Modified G_x payoff}
f(\x,\chi,c) := \min_{j \in J} \sum_{\l \in K} \x^{\l} \sum_{i \in I} \chi^{\l}[i] G^{\l}(i,j).
\end{align}
for every state $\xi \in \dk$ and actions $(\chi,c) \in \Delta(I)^K \times C$. For $\x \in \{\pi^*, \pi^{**}\}$ we set $f(\x, \cdot , \cdot ) = f(\pi, \cdot, \cdot)$. Therefore, $f$ agrees with the reward function $r$ of the MDP described in the previous subsection. 
    
The modified transition rule is described by a lottery $t:\O\times \Delta(I)^K \times C \to \Delta(\O \times \Lambda_1 \times \Lambda_2 )$, where $\Lambda_1$ is the \textit{signal set of Player 1} and $\Lambda_2$ the \textit{signal set of the adversary} in $\G$.\footnote{For the sake of clarity, we note once more that the adversary of Player 1 in $\G$ is not Player 2. As previously discussed, Player 2 is embodied in $\G$ through a device that plays optimal responses in the sense of the payoff function $f$ in \eqref{Modified G_x payoff}.}. The transition rule is formally defined by
    \begin{align*}\label{transition rule}
t(\x,\chi,0) & = \left\{
       \begin{array}{ll}
        (\pi^*,\pi^*,\emptyset) , & \hbox{with prob.\,\,}\,\, x\\ 
        (\x(\chi,i) M, \x(\chi,i) M, \emptyset) , & \hbox{with prob.\,\,}\,\, (1-x)\cdot \chi(\x)[i], \, i \in I
       \end{array}
     \right.
     \\ 
     \text{and}
     \\
t(\x,\chi,1) & = \left\{
       \begin{array}{ll}
        (\pi^*,\pi^*,\emptyset) , & \hbox{with prob.\,\,}\,\, x\\ 
        (\pi^{**},\pi^{**},\emptyset) , & \hbox{with prob.\,\,}\,\, (1-x)
       \end{array}
     \right.,
\end{align*}
for every $\x \in \O$, every $\chi \in \Delta(I)^K$, and $i\in I$, where as before we use the convention that $\x(\chi,i) = \pi(\chi,i)$ for every $\x \in \{\pi^*, \pi^\}$. As in the original version of $\G$, we may take $\Lambda_1 = \O$ and $\Lambda_2 = \emptyset$.

The play description, strategies, $N$-Ces\`aro payoffs, and the motivation behind the modified version of $\G$, remain exactly the same as they were in the original version of $\G$ (see  Subsections \ref{Play of G_x}, \ref{Motivation to G_x}, and \ref{Strategies in G_x}). 

\subsubsection{The Key Role of the Pre-Play Communication in Regards to $\G$}

Up to this point, the possibility of pre-play communication did not play a role in neither the formal  description of the MDP in Subsection \ref{MDP-2} nor in the modification of $\G$ given in Subsection \ref{Subsec:Gx.Modification}. However, the pre-play communication will have a key role, as it will allow Player 1 at any stage $n$ of $\G$, to `split' the state $\x_n \in \O$ prior to taking her $n$'th action $\chi_n$. Such a split embodies the Bayesian update of Player 2's belief over $K$, upon obtaining an informative message from Player 1 via the pre-play communication. In practice, the pre-play communication will serve as the `device' that will allow Player 1 to perform the `memory restorations' in the modified version of $\G$.

In further detail, assume that at state $\x \in \O$ of $\G$, Player 1 first splits $\x$ to $\mu$ using the pre-play communication and thereafter takes the action $\chi \in \Delta(I)^K$.\footnote{We allow $\mu \in \Delta(\Delta(K))$ to be any countable distribution  having mean $\x$. 
} In such case, the action $\chi$ may depend on the element $p \in \text{supp}(\mu)$ which was realized at the pre-play communication.\footnote{For every countably supported probability measure measure $\mu$, denote by $\text{supp}(\mu)$ the support of $\mu$.} Therefore, in such case we write $\chi = \{ \chi\{p\} \in \Delta(I)^K\,:\, p \in \text{supp}(\mu)\}$. Subsequently, the payoff definition on the right-hand side of \eqref{Modified G_x payoff} takes the form
\begin{align*}
    \sum_{p \in \text{supp}(\mu)} \mu[p] \cdot \min_{j \in J} \sum_{\l \in K} p^{\l} \sum_{i \in I} (\chi\{p\})^{\l}\,[i] \, G^{\l}(i,j),
\end{align*}
whereas the transition rule takes the form 
 \begin{align*}
t(\x,\chi,0) & = \left\{
       \begin{array}{ll}
        (\pi^*,\pi^*,\emptyset) , & \hbox{with prob.\,\,}\,\, x\\ 
        (p(\chi\{p\},i) M, p(\chi\{p\},i) M, \emptyset) , & \hbox{with prob.\,\,}\,\, (1-x)\cdot \mu [p]\cdot (\chi\{p\})(\x)[i]
       \end{array}
     \right..
     \end{align*}

     
\subsection{The Deduction of Theorem \ref{Thm RG-1,3}}

We begin with the first item of Theorem \ref{Thm RG-1,3}. The plan of the proof coincides with that given in Subsection \ref{Plan of the Proof} with the exception that $v_{\d}$ is substituted with $V_{\d}$. The probabilistic framework required for the proof is the same as the one described in Subsection \ref{Subsec Probabilistic Framework}. As for the proof of the parallel version of Proposition \ref{prop1} for the modified version of $\G$, the main modification is in the update required for the definition of $\theta^* \in \Theta$ given in Subsection \ref{Subsec Prop 1}. In the current setup, for each integer $m \geq 0$ and integer $0\leq l < \k_{m+1}$, $\theta^*$ prescribes Player 1 to play at stage $\k_0 \oplus \k_m +l$ as follows:
\begin{itemize}
    \item If $\o_{\k_0 \oplus \k_m +l} \neq \pi^{**}$, follow the stationary strategy $\hat{\chi}_{1-x}$ described in Subsection \ref{MDP-2}. That is, take the action $\hat{\chi}_{1-x} (\o_{\k_0 \oplus \k_m +l})$. Note that for $l=0$ it is always the case that $\o_{\k_0 \oplus \k_m +l} = \pi^* \neq \pi^{**}$.
    \item If $\o_{\k_0 \oplus \k_m +l} = \pi^{**}$, first use pre-play communication to split $\o_{\k_0 \oplus \k_m +l}$ to the distribution $\hat{\mu}^*_{l+1} (1-x,\pi)$ described in \eqref{Optimal Priors Dist}. Second, given that $p \in \text{supp}(\hat{\mu}^*_{l+1} (1-x,\pi))$ was realized, take that action $\hat{\chi}_{1-x} (p)$.
\end{itemize}
In this scenario, $\theta^*$ is meant to `restore Player 2's memory' whenever it is erased by the adversary. It does so by splitting $\pi^{**}$ to $\hat{\mu}^*_{l+1} (1-x,\pi)$, being the distribution of Player 2's belief over $K$ at the start of the $(l+1)$'st stage of the repeated game, when Player 1 follows the stationary strategy $\hat{\chi}_{1-x}$. Thus, similarly to the case in the original proof, $\theta^*$ turns the game along stages $\k_0 \oplus \k_m,..., \k_0 \oplus \k_{m+1}-1$ to a random-duration repeated game between Player 1 and 2, and prescribes Player 1 to play optimally in such game to guarantee $V_{1-x}(\pi)/x$. 

In formal terms, we may now restate and reprove Claim \ref{Claim prop1} as follows. \begin{claim}\label{Claim prop1 revised}
 For each integer $m \geq 0$ and $\t \in \T$ we have
\begin{equation*}
E^x_{\theta^{*},\t} \left(\sum_{i = \k_0 \oplus \k_m}^{\k_0 \oplus \k_{m+1} - 1} f(\o_i,\chi_i,c_i)  \right) = \frac{V_{1-x}(\pi)}{x}.
\end{equation*}
\end{claim}
\begin{proof}[Proof of Claim \ref{Claim prop1 revised}]
By the definition of $\theta^*$, $f$, and the optimality of $\hat{\chi}_{1-x}$, as well as the proof details of Claim \ref{Claim Random Duration} we have
\begin{align*}
    & E^x_{\theta^{*},\t} \left(\sum_{i = \k_0 \oplus \k_m}^{\k_0 \oplus \k_{m+1} - 1} f(\o_i,\chi_i,c_i) \, \Bigg|\, \k_0 \oplus \k_m \right)\\
    & = \frac{1}{x}  \sum_{n=1}^{\infty} x\cdot (1-x)^{n-1}  \sum_{p \in \text{supp}(\hat{\mu}^*_{l+1} (1-x,\pi))} \hat{\mu}^*_{l+1}[p] \cdot \min_{j \in J} \sum_{\l \in K} p^{\l} \sum_{i \in I} (\hat{\chi}_{1-x}\{p\})^{\l}\,[i] \, G^{\l}(i,j)\\
    & = \frac{V_{1-x}(\pi)}{x}.
\end{align*}
The above is sufficient to deduce the proof of the claim. 
\end{proof}
Upon establishing the above claim, the proof of the modified version of Proposition \ref{prop1} continues in identical fashion with the slight modification that $\Vert u \Vert_{\infty}$ is replaced with $\Vert G \Vert_{\infty} = \max_{\l,i,j} G^{\l}[i,j]$. This allows one to obtain that  $\gamma_N^x (\theta^*,\t) \geq V_{1-x}(\pi) - o(1)$. 

In addition, the proof of Proposition \ref{prop2} remains valid for the modified version of $\G$ (with the replacement of $\Vert u \Vert_{\infty}$ by $\Vert G \Vert_{\infty}$ and $v_{\d}$ by $V_{1-\d}$), and we obtain that  $\gamma_N^x (\theta,\t^y) \leq V_{1-y}(\pi) + o(1)$, where $\t^y$ is defined in the same way as in \ref{Subsec Prop 2}. Therefore, just as described in Subsection \ref{Plan of the Proof}, one deduces that $\d \mapsto V_{\d}(\pi)$ is non-increasing on $[0,1)$.

As it pertains to the proof of the second item of Theorem \ref{Thm RG-1,3}, we set $V_{\infty} = \lim_{\d \to1^-}  V_{\d}(\p)$, which must exist by the first item of Theorem \ref{Thm RG-1,3}. The proof then follows entirely along the lines of the proof of Theorem \ref{Thm3}, building on the concavity of $V_{\d}$, and replacing $\Vert u \Vert_{\infty}$ by $\Vert G \Vert_{\infty}$ throughout. The only required modification refers to the proof of Lemma \ref{Lemma Thm3 - 1}.  

Specifically, to obtain that $V_{\d}(\x M) \geq V_{\d} (\x) - \frac{(1-\d)}{\d} \Vert G \Vert_{\infty}$, consider the case where the prior is $q_1 = \x M$ and define $\a^* \in \A$ as follows:
    \begin{itemize}
        \item At the first stage split $\x M$ to $\hat{\mu}^*_2 (\d,\x)$ using pre-play communication, where $\hat{\mu}^*_2 (\d,\x)$ was described in \eqref{Optimal Priors Dist}. Then, given that $p \in supp(\hat{\mu}^*_2 (\d,\x))$ was realized, take the optimal action $\hat{\chi}_{\d}(p)$.
        \item Starting from the second stage on follow the optimal stationary strategy $\hat{\chi}_{\d}$.
    \end{itemize}
That is, $\a^*$ is intended to ensure Player 1 the stage payoffs from the second stage on when the prior is $\x$. It does so, by splitting $\x M$ to $\hat{\mu}^*_2 (\d,\x)$, which would have been the prior belief of Player 2 at the second stage when Player 1 follows his optimal strategy $\hat{\chi}_{\d}$ and the repeated game starts from the prior $\x$. Using this definition, one may deduce based on parallel arguments to those in the proof of Lemma \ref{Lemma Thm3 - 1} that $V_{\d}(\x M) \geq V_{\d} (\x) - \frac{(1-\d)}{\d} \Vert G \Vert_{\infty}$. The reverse inequality is achieved by parallel arguments to those given in Eq.\ \eqref{Eq. 3.1}. This concludes the modification to Lemma \ref{Lemma Thm3 - 1} and thus completes the proof.

\appendix
\addcontentsline{toc}{section}{Appendices}
\section*{Appendices}

\section{Case Studies and Calculations}\label{Appendix A}

The following Lemma addresses a selection property of the optimal stationary signaling policy $\chi_{\d}^*$ defined in Subsection \ref{Sub-sub sec Functional Equation}. 

\begin{lemma}\label{lm max}
Assume that $\x\in \dk$ satisfies $u(\x) = (\text{Cav}\, u) (\x)$. Then, for any $\delta \in [0,1)$, $\chi_{\d}^*$ may be selected so that it reveals no information at $\x$.
\end{lemma}

\begin{proof}[Proof of Lemma \ref{lm max}]
    Recall that we have shown in Subsection \ref{Sub-sub sec Functional Equation} that 
    $v_{\d} = (\text{Cav} f_{\d})$, where $f_{\d} := (1-\d) \cdot u + \d \cdot (v_{\d}\circ M)$. Note that by the concavity of $(v_{\d}\circ M)$ on $\dk$, we have that the function $g_{\d} := (1-\d) \cdot(\text{Cav}\, u) + \d \cdot (v_{\d}\circ M)$ is concave on $\dk$ as well. As $f_{\d} \leq g_{\d}$ we obtain that $v_{\d} = (\text{Cav} f_{\d}) \leq g_{\d}$ on $\dk$. On the other hand, as the sender can reveal no information at $\x$,  the recursive formula in \eqref{Belmann2} implies that
    \begin{align*}
        v_{\delta}(\xi) & = \sup_{\chi = \{(\a_s,\x_s)\} \in \mathcal{A_{\xi}}} \sum_{s \in S} \a_s \cdot \bigg\{ (1-\d)\cdot u + \d \cdot ( v_{\d} \circ M )  \bigg\} (\x_s)\\
        & \geq \bigg\{ (1-\d)\cdot u + \d \cdot ( v_{\d} \circ M )  \bigg\} (\x)\\
        & = \bigg\{ (1-\d)\cdot (\text{Cav}\,u) + \d \cdot ( v_{\d} \circ M )  \bigg\} (\x) = g_{\d} (\x).      
    \end{align*}
    Combining the two we obtain that 
    \begin{align*}
        \d_{\{\x\}} \in \text{arg\, max}_{\chi = \{(\a_s,\x_s)\} \in \mathcal{A_{\xi}}} \sum_{s \in S} \a_s \cdot \bigg\{ (1-\d)\cdot u + \d \cdot ( v_{\d} \circ M )  \bigg\} (\x_s).
    \end{align*}  
    Therefore, the optimal stationary strategy $\chi_{\d}^*$ may be chosen to satisfy $\chi_{\d}^*(\x) = \d_{\{\x\}}$, i.e., $\chi_{\d}^*$ splits $\x$ to itself.
\end{proof}

We now move on to examples of the possible phenomena presented in Figure \ref{figure 1}. First, by Theorem \ref{Thm3} in the case where $M$ is ergodic (i.e., irreducible and aperiodic), $v_{\delta}$ converge uniformly (over $\dk$) to a constant function. Therefore, in such case the relations between  $\Phi$ and $\Psi$ obey either Case A or Case B in Figure \ref{figure 1}. Moreover, both of those cases may occur. Trivially, for constant $u$, we have case B where $\d_0$ described in Corollary \ref{Cor1} equals to $0$. As an example for the possibility of case A, consider $K=\{1,2\}$ and let $M$ be defined by
\begin{equation*} 
M=\begin{pmatrix}
  \ 1/2& 1/2\ \\
 \ 1/6 & 5/6\
\end{pmatrix}.
\end{equation*}
so that $\pi_M = (0.25,0.75)$. Also, let $u((p,1-p)) := p(2-3|p-\frac{1}{2}|)+(1-p)p/10$ where $p \in [0,1]$. It was shown in an earlier working paper version\footnote{See Example 4 on p.\ 17-18 in \href{https://www.math.tau.ac.il/~lehrer/Papers/Markovian\%20Persuasion.pdf}{https://www.math.tau.ac.il/~lehrer/Papers/Markovian\%20Persuasion.pdf}} of Lehrer and Shaiderman \emph{\cite{MP}} that 
\begin{align*}
& v_{\d}((p,1-p)) =\\ & = \left\{
       \begin{array}{ll}
        2.05\left(3(\d-1)p - \d/2\right)/(\d -3) , & \hbox{for\,\,}\,\, p \in [0,0.5],\\ 
        (1-\d)u((p,1-p)) + \d\cdot 2.05\left((\d-1)p - 1/2\right)/(\d -3) , & \hbox{for\,\,}\,\, p \in [1/2,1].
       \end{array}
     \right.
\end{align*}
In such a case it can be verified that $\Phi > \Psi$ on $[0,1)$, thus suiting Case A. Moreover, the function $\Psi$ is strictly increasing on $[0,1)$, whereas $\Phi$ is constant and equals $\cavu(\p)$ on $[0,1)$. 

As it turns out, Case C in Figure \ref{figure 1} is possible as well. Indeed, take the periodic irreducible matrix $M$ defined by
\begin{equation*}
  M=\begin{pmatrix}
\ 0& 1\ \\
 \ 1 & 0\
\end{pmatrix},
\end{equation*}
and take $u((p,1-p)) = p(1-p)$ for every $p \in [0,1]$. By the definition of $M$ and $u$ we have that $\Psi(\delta)=0$ for every $\delta$. Nevertheless, since $\pi_M = (0.5,0.5)$, and $u$ is concave, it follows from Lemma \ref{lm max} that $\Phi(\delta) =  u(\pi_M) =0.25$ for every $\delta \in [0,1)$.

\section{Sorin's Proof}\label{Appendix B}

The current section provides a detailed version of Sorin's proof for the Aumann Maschler model, which corresponds to the case where $M = \text{Id}_k$ in the framework of Subsection \ref{Subsec Repeated Game}. Assume that the game starts from the prior $q \in \dk$.

For any sequence of numbers $\{a_n\,:\, n\geq 1\}$ and $0< \lambda \leq 1$ define:
\begin{equation*}
    F_\lambda^m(\{a_n\}) := \sum_{n=1}^{\infty} \lambda (1 - \lambda)^{n-1} a_{m+n}.
\end{equation*}
Note that $F_\lambda^m(\{a_n\})$ equals to the $(1-\lambda)$-discounted valuation of the shifted sequence $\{a_{m+1},a_{m+2},...\}$. The key algebraic property used by Sorin is the following: Assume that $0<\mu < \lambda\leq 1$. Then one has that 
\begin{align}\label{Eq. Algebraic Relation}
F_\mu^0(\{a_n\}) = \frac{\mu}{\lambda} F_\lambda^0(\{a_n\}) + \frac{\mu}{\lambda} (\lambda - \mu) \sum_{m=0}^{\infty} (1 - \mu)^m F_\lambda^{m+1}(\{a_n\}).
\end{align}
In particular, the $(1-\mu)$-discounted valuation of $\{a_n\}$ can be decomposed to a convex combination of the $(1-\lambda)$-discounted valuations of the shifted sequences $\{(a_m,a_{m+1},...)\,|\, m \in \N \}$.

Next, let $\b_{\lambda} (q) \in \mathcal{B}$ denote an optimal strategy for Player 2 in the game with discount factor $(1-\lambda)$. 

\begin{lemma}\label{Claim Sorin}
    Let $\a \in \mathcal{A}$. Define $a_n := E^q_{\a,\b_{\lambda} (q)}\left[ G^{X_n}(i_n, j_n) \right]$ for every $n\geq 1$. Then for every $m\geq 0$ it holds that:
    \begin{align*}
        F_\lambda^m(\{a_n\})  \leq V_{1-\lambda}(q).
    \end{align*}
\end{lemma}

\begin{proof}[Proof of Lemma \ref{Claim Sorin}]
    For $m=0$ the inequality is immediate from the optimality of $\b_{\lambda} (q)$, and the definition of  $F_\lambda^0$ and $\{a_n\}$. Next, consider the case where $m\geq 1$ and assume by contradiction that $F_\lambda^m(\{a_n\})  > V_{1-\lambda}(q)$. Consider the following strategy $\b^* \in \mathcal{B}$ of Player 2: along stages $1,...,m$ follow $\b_{\lambda} (q)$; then, starting from stage $m+1$, forget all previous information, and start playing $\b_{\lambda} (q)$ anew. 

    By the stationarity of the underlying dynamics of the game we have that:
    \begin{align}\label{Eq. Sorin}
       E^q_{\a,\b^*}\left[ \sum_{n=1}^{\infty} \lambda (1-\lambda)^{n-1}  G^{X_{m+n}}(i_{m+n}, j_{m+n}) \right] \leq V_{1-\lambda}(q) < F_\lambda^m(\{a_n\}). 
    \end{align}
    On the other hand, as by definition $a_n = E^q_{\a,\b^*}\left[ G^{X_n}(i_n, j_n) \right]$ for every $n=1,...,m$ we obtain that with the help of \eqref{Eq. Sorin} that
    \begin{align*}
        & E^q_{\a,\b^*}\left[ \sum_{n=1}^{\infty} \lambda (1-\lambda)^{n-1}  G^{X_{n}}(i_{n}, j_{n}) \right]  = F_\lambda^0(\{a_n\})\\
        & \quad = \sum_{n=1}^m \lambda (1-\lambda)^{n-1} a_n + (1-\lambda)^{m} F_\lambda^m(\{a_n\})\\
        & \quad  > \sum_{n=1}^m \lambda (1-\lambda)^{n-1} E^q_{\a,\b^*}\left[ G^{X_n}(i_n, j_n) \right]\\
        & \quad \quad + (1-\lambda)^{m} E^q_{\a,\b^*}\left[ \sum_{n=1}^{\infty} \lambda (1-\lambda)^{n-1}  G^{X_{m+n}}(i_{m+n}, j_{m+n}) \right]\\
        & \quad = E^q_{\a,\b^*}\left[ \sum_{n=1}^{\infty} \lambda (1-\lambda)^{n-1}  G^{X_{m+n}}(i_{n}, j_{n}) \right].
    \end{align*}
    As the inequality is strict, and holds for every $\a \in \mathcal{A}$, it in particular holds for the optimal strategy $\a_{1-\lambda}(q)$ of Player 1 in the game with discount factor $1-\lambda$. For such a strategy the above inequality suggests that:
    \begin{align*}
        V_{1-\lambda}(q) > E^q_{\a_{1-\lambda}(q),\b^*}\left[ \sum_{n=1}^{\infty} \lambda (1-\lambda)^{n-1}  G^{X_{m+n}}(i_{n}, j_{n}) \right].
    \end{align*}
    However, as by the definition of $\a_{1-\lambda}(q)$ we have 
    \begin{align*}
        E^q_{\a_{1-\lambda}(q),\b^*}\left[ \sum_{n=1}^{\infty} \lambda (1-\lambda)^{n-1}  G^{X_{m+n}}(i_{n}, j_{n}) \right] \geq V_{1-\lambda}(q),
    \end{align*}
    thus arriving at a contradiction. Therefore, $F_\lambda^m(\{a_n\})  \leq V_{1-\lambda}(q)$ for every $m\geq 1$ as required.
\end{proof}

Next, let us fix some $\a \in \mathcal{A}$ and consider once again the sequence $a_n := E^q_{\a,\b_{\lambda} (q)}\left[ G^{X_n}(i_n, j_n) \right]$. Using Lemma \ref{Claim Sorin} and relation \eqref{Eq. Algebraic Relation} we obtain that 
\begin{align*}
F_\mu^0(\{a_n\}) & = \frac{\mu}{\lambda} F_\lambda^0(\{a_n\}) + \frac{\mu}{\lambda} (\lambda - \mu) \sum_{m=0}^{\infty} (1 - \mu)^m F_\lambda^{m+1}(\{a_n\})\\
& \leq \frac{\mu}{\lambda} V_{1-\lambda}(q) + \frac{\mu}{\lambda} (\lambda - \mu) \sum_{m=0}^{\infty} (1 - \mu)^m V_{1-\lambda}(q) \\
& = V_{1-\lambda}(q).
\end{align*}
As the above was true for any $\a \in \mathcal{A}$, by the definition of $F_\mu^0$, we obtain that $V_{1-\mu}(q) \leq V_{1-\lambda}(q)$ for every $\mu < \lambda$. This suffices for the proof of the monotonicity in the Aumann and Maschler model.

Consider now the case where $M \neq \text{Id}_k$ and let $\pi$ be an invariant distribution of $M$. The proof above can be duplicated, in the same manner with $q=\pi$. Indeed, the argument in the proof of Lemma \ref{Claim Sorin} according to which Player 2 can start following $\b_{\lambda} (\pi)$ starting from stage $m+1$ to guarantee that 
\begin{align*}
       E^{\pi}_{\a,\b^*}\left[ \sum_{n=1}^{\infty} \lambda (1-\lambda)^{n-1}  G^{X_{m+n}}(i_{m+n}, j_{m+n}) \right] \leq V_{1-\lambda}(\pi), \quad \forall \a \in \A, 
    \end{align*}
only requires that starting from stage $m+1$, the dynamics of the states are the same as those starting from the first stage. However, this holds, as the Markov chain's distribution over $K$ at stage $m+1$, is just $\pi M^m = \pi$. Therefore, the result for the general case follows by the same arguments as before. 
\end{document}